\documentclass[11pt]{amsart}
\usepackage{amsmath}
\usepackage{amssymb, verbatim}

\usepackage{amssymb,amscd,amsmath,amsthm,url}
\usepackage{pstricks,pst-node,pst-plot}
\usepackage{comment}
\usepackage{color}
\usepackage{mathtools}
\usepackage{verbatim}
\usepackage{enumerate}
\usepackage[all]{xy}
\usepackage{breqn}
\usepackage{tikz}
\usepackage{tikz-cd}
\usepackage{hyperref}

\numberwithin{equation}{section}

\def \R {\mathbb R}
\def \C {\mathbb C}

\def \eps {\varepsilon}
\def \bj {\bar{j}}

\def \Vol {\text{Vol}}
\def \Ric {\text{Ric}}

\def \Re {\text{Re}}
\def \db {\bar{\partial}}
\def \d {\partial}
\def \Tr {\text{Tr}}

\def \lapl {\Delta}

\newcommand{\Bl}{\textrm{Bl}}

\newcommand{\del}{\partial}

\newcommand{\dbar}{\overline{\del}}
\newcommand{\ddb}{\sqrt{-1}\del\dbar}

\newcommand{\Spec}{\text{Spec }}

\renewcommand{\d}{\partial}

\renewcommand{\leq}{\leqslant}
\renewcommand{\geq}{\geqslant}

\renewcommand{\epsilon}{\varepsilon}
\renewcommand{\phi}{\varphi}

\def\XXint#1#2#3{{\setbox0=\hbox{$#1{#2#3}{\int}$ }
\vcenter{\hbox{$#2#3$ }}\kern-.6\wd0}}

\newcommand{\dd}[1] {\frac{\partial}{\partial #1 }}

\newtheorem {Theorem} {Theorem}[section]
\newtheorem {Proposition}[Theorem]      {Proposition}
\newtheorem {Lemma}      [Theorem]       {Lemma}
\newtheorem {Corollary} [Theorem] {Corollary}

\newtheorem {Conjecture}{Conjecture}
\newtheorem {Assumption}{Assumption}

\theoremstyle{definition}
\newtheorem  {Definition}{Definition}[section]

\theoremstyle{remark}
\newtheorem{Remark}	{Remark}

\title{On the degeneration of asymptotically conical Calabi-Yau metrics}
\date{\today}
\author[T. C. Collins]{Tristan C. Collins}
  \email{tristanc@mit.edu}
  \address{Department of Mathematics, Massachusetts Institute of Technology, 77 Massachusetts Avenue, Cambridge, MA 02139}
 \thanks{T.C.C is supported in part by NSF grant DMS-1944952, DMS-1810924 and an Alfred P. Sloan Fellowship. }

\author[B. Guo]{Bin Guo}
\email{bguo@rutgers.edu}
\address{Department of Mathematics and Computer Science, Rutgers University-- Newark, 101 Warren Street, Newark, NJ 07102}
\thanks{B. G. is supported in part by NSF grant DMS-1945869}

\author[F. Tong]{Freid Tong}
\email{tong@math.columbia.edu}
\address{Department of Mathematics, Columbia University, 2990 Broadway, New York, NY 10027}
\thanks{F. T. is supported in part by NSF grant DMS-1855947}

\begin{document}

\maketitle

\begin{abstract}
We study the degenerations of asymptotically conical Ricci-flat K\"ahler metrics as the K\"ahler class degenerates to a semi-positive class. We show that under appropriate assumptions, the Ricci-flat K\"ahler metrics converge to a incomplete smooth Ricci-flat K\"ahler metric away from a compact subvariety. As a consequence, we construct singular Calabi-Yau metrics with asymptotically conical behaviour at infinity on certain quasi-projective varieties and we show that the metric geometry of these singular metrics are homeomorphic to the topology of the singular variety. Finally, we will apply our results to study several classes of examples of geometric transitions between Calabi-Yau manifolds. 
\end{abstract}

\section{Introduction}
Following Yau's resolution of the Calabi Conjecture \cite{Yau78} the study of Ricci-flat K\"ahler metrics has played a central role in geometric analysis.  Subsequently,  motivated by questions in differential geometry, mathematical physics, and algebraic geometry there has been a great deal of interest in extensions of Yau's theorem to the complete, non-compact setting \cite{TY90, TY91, CHNP, HHN, Joyce, VaC1, VaC2, Go, CH13, CH15, CH3, Li, Fau, BK, BK2}, the degeneration of Calabi-Yau metrics (see, for example, the surveys \cite{ToSurv11, ToSurv12, ToSurv19} and the references there in), and the existence of Calabi-Yau metrics on singular spaces (see for example \cite{EGZ07, JS14}). In this paper we initiate the study of degenerations of {\em non-compact} Calabi-Yau manifolds, and the existence of Calabi-Yau metrics on certain non-compact singular varieties.

In the compact setting, a special class of Calabi-Yau degenerations are obtained by degenerating the K\"ahler class.  More precisely, fix a compact Calabi-Yau manifold $X$, and let $\mathcal{K} \subset H^{1,1}(X,\mathbb{R})$ denote the K\"ahler cone, consisting of all $(1,1)$ cohomology classes admitting a K\"ahler representative; $\mathcal{K}$ is an open convex cone in $H^{1,1}(X,\mathbb{R})$.  For each class $[\omega]\in \mathcal{K}$, Yau's theorem \cite{Yau78} yields the existence of a unique Ricci-flat K\"ahler metric $\omega_{CY} \in [\omega]$.  Choose a family of K\"ahler classes $[\omega_t]\in \mathcal{K}, t\in (0,1]$ such that $[\omega_t ]\rightarrow [\alpha]\in \partial \mathcal{K}$ as $t\rightarrow 0$.  We are interested in understanding the geometry of $(X,\omega_{t,CY})$ as $t\rightarrow 0$.  Roughly speaking this question can be divided into two cases; the collapsing case, when $\int_{X}\alpha^{n}=0$, and the non-collapsing case, when $\int_{X}\alpha^n >0$. The non-collapsing case, is reasonably well understood, thanks to work of Tosatti \cite{To07} and the first author and Tosatti \cite{CT}.

One way to construct a non-collapsed family of Calabi-Yau manifolds is as follows; suppose $X_0$ is a normal, Gorenstein, projective variety with $K_{X_0}$ trivial.  Suppose that $\pi: X\rightarrow X_0$ is a crepant resolution of singularities, and let $[\alpha] = \pi^{*}[\omega_0]$ for some K\"ahler class $[\omega_0]$ on $X_0$.  A family of K\"ahler classes on $X$ converging to $[\alpha]$ gives rise to non-collapsed family of Calabi-Yau metrics.  In this case, the results of \cite{CT} say that the Calabi-Yau metrics $\omega_{t,CY}$ converge in $C^{\infty}_{loc}(X\backslash {\rm Exc}(\pi))$, to an incomplete metric $\omega_{0, CY}$ and $(X,\omega_{t,CY})$ Gromov-Hausdorff converge to the completion $\overline{(X\backslash {\rm Exc}(\pi), \omega_{0,CY})}$. $\omega_{0,CY}$ descends to a Ricci-flat metric on $X_0^{reg}$, and one can ask whether the metric geometry of $\omega_{0,CY}$ is related to the geometry of the $X_0$. In this case, assuming that $[\alpha]\in H^{1,1}(X,\mathbb{Q})$, Song \cite{JS14}, proved that $\overline{(X\backslash {\rm Exc}(\pi), \omega_{0,CY})} = \overline{(X_0^{reg}, \omega_{0,CY})}$ is homeomorphic to $X_0$.  In particular, this yields the existence of a natural Calabi-Yau metric on the singular variety $X_0$.  
	 
	In this paper we study degenerations of Calabi-Yau metrics on complete {\em non-compact} Calabi-Yau manifolds asymptotic to a cone. Complete, non-compact Calabi-Yau manifolds were first constructed by Tian-Yau in \cite{TY90, TY91}, and a plethora of examples are now known to exist. A particular subset of these are Calabi-Yau manifolds which are asymptotic to a cone at infinity, these are sometimes called asymptotically conical Calabi-Yau manifolds. Conical Calabi-Yau manifolds are of fundamental importance, since they arise as blow-up limits at the singular points in the limit of a non-collapsing family of K\"ahler-Einstein manifolds (or more generally K\"ahler manifolds with bounded Ricci curvature).  The conical asymptotics should be regarded here as akin to the non-collapsing condition in the setting of compact Calabi-Yau manifolds discussed above.
	 
	 The first analytic construction of asymptotically conical Calabi-Yau manifolds was given in \cite{TY91} and \cite{BK, BK2}, and the construction has been further refined by the work of many authors, see \cite{Joyce, VaC1, VaC2, Go, CH13, CH15, CH3, Li, Fau} and the references therein. One nice improvement given by these refinements is that, in analogy with Yau's theorem in the compact case \cite{Yau78}, one is able to produce an asymptotically conical Ricci-flat K\"ahler metric in every suitable K\"ahler class on an asymptotically conical K\"ahler manifold $X$.  In particular, this yields families of degenerating asymptotically conical Ricci-flat K\"aher metrics, and one can then ask what properties limits of these spaces possess.
	 
	   The motivation for studying these limits is two-fold.  First, there is a broad class of non-compact examples which are expected to model the local behavior of Calabi-Yau metrics on compact Calabi-Yau manifolds near certain singular limits. Understanding the behavior of these ``local" models through singular transitions will help to sharpen our understanding of the degeneration of Ricci-flat metrics in the compact setting. Secondly, understanding these metric limits allows us to prove the existence of asymptotically conical Calabi-Yau metrics on singular spaces.  These metrics can be viewed as interpolating between affine varieties with conical Calabi-Yau metrics (or equivalently, Sasaki-Einstein manifolds).

Let us describe the set-up under consideration and state our main theorems. The terminologies used in this section will be explained in the next section.  Let $(X, J, \omega, \Omega)$ be an open K\"ahler manifold with trivial canonical bundle, with only one end which is asymptotic to a Calabi-Yau cone $(C, J_C, \omega_C, \Omega_C)$ with rate $\nu>0$. Consider a linear family of $\nu$-almost compactly supported K\"ahler classes $[\alpha_t] = (1-t)[\alpha_0]+t[\alpha_1]\in H^{1, 1}_{\nu}(X)$ for $t\in (0, 1]$. Suppose $[\alpha_0]$ satisfies the following assumption.
	
	\begin{Assumption}\label{ass: mainAss}
		$[\alpha_0]$ contains a semi-positive form $\alpha_0$, and there exists $\eps_0>0$ and a $\psi \in PSH(X, \alpha_0)$ such that $\alpha_0+i\d\db\psi \geq  \eps_0\omega$ for some K"ahler form $\omega$ on $X$.  Furthermore, assume that  $\psi$ is smooth away from a compact analytic subvariety $V\subset X$, and $V= \{\psi= -\infty\}$. 
	\end{Assumption}
	
	\begin{Remark}\label{rk: Ass1}
	We expect that Assumption~\ref{ass: mainAss} essentially always applies, possibly after weakening the semi-positivity assumption.  In fact, in analogy with the main result of \cite{CT},  we expect that
	\[
	V = \bigcup_{Y\subset X : \int_{Y}\alpha_0^{\dim Y}=0}Y
	\]
	where the union is taken over compact, irreducible analytic subvarieties.  We will prove this in a large class of examples;  see the discussion in Section~\ref{sec: Null=Enk}.
	\end{Remark}

 In \cite{CH13}, it is proved that for $t\in (0, 1]$ there exists a unique asymptotically conical Ricci-flat K\"ahler metric $\omega_{t, CY}\in [\alpha_t]$ satisfying the complex Monge-Amp\`ere equation
	\begin{equation*}
		\omega_{t, CY}^n = i^{n^2}\Omega\wedge\bar{\Omega}
	\end{equation*}
Our first theorem is the following, 
\begin{Theorem}\label{Main Theorem}	
	Let $0<\nu<2n$ and consider a linear family of $\nu$-almost compactly supported K\"ahler classes $[\omega_t] = (1-t)[\alpha_0]+t[\omega]\in H^{1, 1}_{\nu}(X, \R)$ for $t\in (0, 1]$. Suppose $[\alpha_0]$ satisfies Assumption~\ref{ass: mainAss}.  Let $\omega_{t,CY}$ be the asymptotically conical Calabi-Yau metrics in $[\omega_t]$.  Then, as $t\to 0$, the Ricci-flat K\"ahler metrics $\omega_{t, CY}$ converge in $C^{\infty}_{loc}(X\setminus V)$ to an incomplete Ricci-flat K\"ahler metric $\omega_{0, CY}$ on $X\setminus V$ satisfying
	\begin{equation}\label{eq: mainTheoremCMA}
		\omega_{0, CY}^n = i^{n^2}\Omega\wedge\bar{\Omega}.
	\end{equation}
	Moreover, we have
	\begin{enumerate}
	\item $\omega_{0, CY}$ extends across $V$ as a positive current with locally bounded potentials and~\eqref{eq: mainTheoremCMA} holds globally in the sense of Bedford-Taylor \cite{BT}.
	\item $\omega_{0, CY}$ is asymptotically conical at infinity and, outside of a compact set $K\subset X$, $\omega_{0, CY}$ satisfies $|\nabla^k(\omega_{0, CY}-\omega_C)|_{\omega_{C}} = O(r^{-\nu-k})$, where $r(x)= {\rm dist}(x_0,x)$ is the distance to a fixed point with respect to the conical K\"ahler metric $\omega_{C}$.
	\item $\omega_{0,CY}$ is unique in the sense that, if $\omega$ is any closed positive current in the class $[\omega_{0,CY}]$ with locally bounded potentials, which is smooth on $X\backslash V$, asymptotically conical at any rate $\delta>0$, and satisfying~\eqref{eq: mainTheoremCMA} on $X$ in the sense of Bedford-Taylor, then $\omega = \omega_{0,CY}$.
	\end{enumerate}
\end{Theorem}

The reader may wish to compare this result with the analogous result in the compact case \cite[Theorem 1.6]{CT}.  As discussed before, a natural way to construct examples where Theorem~\ref{Main Theorem} applies is to consider resolutions of singular varieties. 

\begin{Theorem}\label{Theorem2}
Let $(X_0, \Omega)$ be a normal, log-terminal, Gorenstein variety with $K_{X_0}$ trivial, and suppose that $X_0$ has compactly supported singularities and admits a crepant resolution of singularities $\pi: (X, \Omega) \to (X_0, \Omega)$.  Suppose that  $L\to X_0$ is an ample line bundle on $X_0$ (see Section~\ref{sec: metricGeometry} for the definition of ampleness in this context). Let $[\alpha_0] = \pi^{*}c_1(L) \in H^2(X, \R)$ and suppose that $(X, J, \omega, \Omega)$ and $[\omega_t] = (1-t)[\alpha_0]+t[\omega]\in H^{1, 1}_{\nu}(X, \R)$ is a family of K\"ahler classes satisfying the same hypothesis as in Theorem \ref{Main Theorem}. (In particular $[\alpha_0]$ satisfies Assumption~\ref{ass: mainAss})
	In the situation above the singular Ricci-flat current $\omega_{0,CY}$ descends to a Ricci-flat K\"ahler current on $X_0$ and satisfies
	\begin{enumerate}
		\item $\omega_{0, CY}$ is a smooth Ricci-flat K\"ahler metric on $\pi^{-1}(X_0^{reg})$. 
		\item $\omega_{0, CY}$ descends to a K\"ahler current on $X_0$, (i.e. $\omega_{0,CY}\geq \omega$ for some smooth K\"ahler form $\omega$ on $X_0$)
		\item $\overline{(X_0^{reg}, \omega_{0, CY})}$ is homeomorphic to $X_0$. 
		\item $(X,\omega_{t, CY})$ Gromov-Hausdorff converges to $X_0$ with the distance function induced by $ \omega_{0, CY}$.
	\end{enumerate}
\end{Theorem}

A couple of remarks are in order concerning the assumptions of Theorem~\ref{Theorem2}
\begin{Remark}
\begin{enumerate}
\item Theorem~\ref{Theorem2} requires that Assumption~\ref{ass: mainAss} to hold for the class $[\alpha_0]$. As pointed out in Remark~\ref{rk: Ass1}, we expect that in this situation that we can always take $V=\pi^{-1}(X_0^{sing})$, and we will prove this is a large number of cases in Lemma~\ref{lem: DP}. Although we don't actually need to assume this for the proof of Theorem~\ref{Theorem2}. 
\item The assumption on the existence of an ample line bundle $L$ may seem at odds with our discussion earlier in the introduction.  In many cases where Theorem~\ref{Theorem2} applies, we will take $L= \mathcal{O}_{X_0}$.  This can be done, for example, when $X_0$ is affine which is a natural setting for studying  Calabi-Yau varieties with isolated singularities.
\end{enumerate}
\end{Remark}

We apply these results to study several classes of examples.  Let us briefly describe one particular class.  Consider the quasi-homogeneous affine variety
\[
Y_{p,q} := \{xy +z^p-w^q=0\} \subset \mathbb{C}^4,
\]
 where without loss of generality we can assume $p\leq q$. $Y_{p,q}$ is normal, Gorenstein and log-terminal, and by \cite{CS} $Y_{p,q}$ admits a conical Calabi-Yau metric if and only if $q<2p$.  A result of Katz \cite{Katz} says that the $Y_{p,p}$ admits $p$ inequivalent small (and hence crepant) resolutions resolutions $\mu_i: Y^i\rightarrow Y_{p,p}$ (and if $p\ne q$ then no small resolution exists). We therefore have the following picture
  \[
\begin{tikzcd}
Y^{1} \arrow[rrd, "\mu_{1}" ] & Y^{2} \arrow[rd, "\mu_2"]& \cdots & Y^{p-1} \arrow[ld, "\mu_{p-1}" '] & \arrow[lld, "\mu_p" '] Y^p\\
\, & \,& Y_{p,p}&\, &\,
\end{tikzcd}
\]
with each pair $Y^i, Y^j$ related by a flop.  When $p=2$, this is the well-known example of the Atiyah flop \cite{Atiyah}. In Section~\ref{sec: Examples} we apply our results to this setting. 

\begin{Corollary}
Let $Y^{i}$ be a small resolution of the $Y_{p,p}$ singularity, and let $\omega_0$ denote the Calabi-Yau metric on $Y_{p,p}$.  Let $[\omega_t] := (1-t)[\alpha_0] + t[\omega]$ be any linear family of K\"ahler classes on $Y^{i}$, where $[\alpha_0] \in H^{1,1}(Y^{i}, \mathbb{Q})$ is not K\"ahler.  Then for all $t>0$ there is an asymptotically conical Calabi-Yau metric $\omega_{t,CY} \in [\omega_t]$.  Furthermore, there is a partial resolution $\bar{\mu}_i:\overline{Y} \rightarrow Y_{p,p}$ and a map $\nu: Y^i \rightarrow \overline{Y}$ such that the following diagram commutes
\[
\begin{tikzcd}
Y^{i} \arrow[r, "\nu"] \arrow[rd, "\mu_i" ']& \overline{Y} \arrow[d, "\bar{\mu}_i"]\\
 & Y_{p,p}
\end{tikzcd}
\]
As $t\rightarrow 0$, $\omega_{t, CY}$ converge in $C^{\infty}_{loc}(Y^i\backslash {\rm Exc}(\nu))$ to an incomplete, asymptotically conical Calabi-Yau metric $\overline{\omega}$ on $\overline{Y}_{reg}$ and $(Y^i, \omega_t)$ Gromov-Hausdorff converges to $\overline{(\overline{Y}_{reg}, \overline{\omega})}$ which is homeomorphic to $\overline{Y}$.  Furthermore, if $[\alpha_0]=0$ then $\overline{Y}= Y_{p,p}$, $\bar{\mu}_i$ is the identity, and $\bar{\omega} = \omega_0$ the Calabi-Yau metric on $Y_{p,p}$.  In particular, when $[\alpha_0]=0$, for any $i,j$ the flop from $Y^i$ to $Y^j$ is continuous in the Gromov-Hausdorff topology in the sense that
\[
\begin{tikzcd}
(Y^{i}, \cdot\omega_{t,CY}) \arrow[r, "GH"] &(Y_{p,p}, \omega_{0})& \arrow[l,"GH" '] (Y^{j}, \cdot \omega_{t,CY})
\end{tikzcd}
\]
\end{Corollary}

A second general class of examples we consider gives rise to the following specific example. Let $X$ be a del Pezzo surface of degree $d \geq 2$, and let $\tilde{X} = \Bl_pX$ be the blow-up at a point $p\in X$.  Assume that $\tilde{X}$ is Fano (and if $d=8$ assume that $\tilde{X}$ is toric).  Then the canonical cone  
\[
C:= \Spec \bigoplus_{m\geq 0} H^{0}(\tilde{X}, -K_{\tilde{X}}^{\otimes m})
\]
admits a conical Calabi-Yau metric \cite{MSYau, Tian90, TY87, FOW, CS}.  Then we prove
\begin{Corollary}
In the above setting, there is an asymptotically conical Calabi-Yau metric on the relative spectrum $Z:= \underline{{\Spec}}(K_{X}\otimes \mathfrak{m}_p)$ which is asymptotic at infinity to the conical Calabi-Yau metric on $C$.
\end{Corollary}

The metric on $Z$ is constructed as a limit of asymptotically conical Calabi-Yau metrics on a small resolution, and we again obtain a Gromov-Hausdorff covergence statement; see Section~\ref{sec: Examples} for a complete discussion.  

We will explain a speculative picture in which that space $Z$ can be viewed as a cobordism between Sasakian manifolds; in this case, the link of the $A_1$ singularity (topologically $S^2\times S^3$) and the link of the cone $C$ (topologically $\#(9-d+1)S^2\times S^3$).  The Calabi-Yau metric on $Z$ upgrades this to a cobordism of Sasaki-Einstein manifolds.  In this picture the volume of the geodesic spheres can be viewed as a sort of Morse function.

The examples above all come from (partial-)resolutions of Calabi-Yau cones. Our theorem can also yield examples where the complex structure at $\infty$ is not biholomorphic to the asymptotic cone. 

Let $X$ be an asymptotically conical Calabi-Yau manifold, then by \cite{Gra}, there exist a normal Stein space $Y$ with finitely many isolated singularities and there is a holomorphic map $\pi: X\to Y$ with connected fibers, is an biholomorphim outside the singularities of $Y$ and $\pi^{\star}\mathcal{O}_Y = \mathcal{O}_X$. The map $\pi$ contracts the maximal compact analytic subset of $X$ and $Y$ is called the Remmert reduction of $X$. Since $Y$ is a stein space, it properly embeds into $\C^N$ for some $N$ sufficiently large. The singularities of $Y$ are rational \cite[Theorem A.2]{CH13}, and hence Cohen-Macaulay, and since $K_X$ is trivial and $Y$ is normal, it follows that $K_Y$ is trivial and $Y$ is Gorenstein. Hence $\pi$ is a crepant resolution of $Y$. 

\begin{Corollary}
	Assuming Assumption~\ref{ass: mainAss} holds for $[\alpha_0] = 0$, applying our theorem with $[\alpha_0] = 0\in H^2(X, \R)$, $\omega_{0, CY}$ descends to a singular CY current on $Y$ and the AC Calabi-Yau metrics $\omega_{t, CY}$ in the classes $t[\omega]\in H^2(X, \R)$ Gromov-Haussdorff coverge to the Remmert reduction $Y$. 
\end{Corollary}

The outline of this paper is as follows.  In Section~\ref{sec: Prelim} we discuss some basic properties of asymptotically conical K\"ahler manifolds, and state two main propositions (Propositions~\ref{Background metrics}, and~\ref{estimates}).  We give the proof of Theorem~\ref{Main Theorem} assuming these two propositions.  In Section~\ref{sec: background} we discuss the construction of good background metrics, and prove Proposition~\ref{Background metrics}.  In Section~\ref{sec: estimates} we prove some a priori estimates and deduce Proposition~\ref{estimates}, completing the proof of Theorem~\ref{Main Theorem}.  In Section~\ref{sec: metricGeometry} we use $L^2$ estimates to prove Theorem~\ref{Theorem2}.  Finally, in Section~\ref{sec: Examples} we explain examples where Theorems~\ref{Main Theorem} and~\ref{Theorem2} are applicable, and discuss a speculative Morse theoretic picture.  
\bigskip

{\bf Acknowledgements:}. The authors are grateful to D. H. Phong for his interest and encouragement.  The authors are also grateful to R. Conlon and H.-J. Hein for explaining aspects of their papers \cite{CH13, CH15, CH3}.

\section{Preliminaries}\label{sec: Prelim}

\subsection{Asymptotically conical K\"ahler manifolds}
We quote here some basic definitions and an existence theorem for asymptotically conical Calabi Yau metrics from \cite{CH13}. 

\begin{Definition}\label{Kahler cone}
	\begin{enumerate}
		\item[(A)] An open {\it K\"ahler cone} $(C, J_C, \omega_C, g_C)$ is a Riemannian cone $(C, g_C)$ with smooth link $L$ that is additionally equipped with a complex structure $J_C$ such that the K\"ahler form is $\omega_C = i\d\db r_{C}^2$ where $r_{C}$ is the distance function from the tip of the cone. 
		\item[(B)] A {\it Calabi-Yau cone} $(C, J_C, \omega_C, g_C, \Omega_C)$ is a K\"ahler cone with an additional holomorphic volume form $\Omega_C$ such that $\omega_C^n = i^{n^2}\Omega_C\wedge\bar{\Omega}_C$. 
	\end{enumerate}
\end{Definition}

\begin{Definition}\label{AC Kahler}
\begin{enumerate}
	\item[(A)] A K\"ahler manifold $(X, J, g, \omega)$ is called {\it asymptotically conical} if there exist a K\"ahler cone $(C, J_C, g_C, \omega_C)$ and a diffeomorphism $\Phi:C\setminus B_R(o)\to X\setminus K$ for some $K\subset \subset X$  and $o$ is the vertex of the cone $C$, and $\nu>0$  such that the following hold
	\begin{equation*}
		|\nabla^k(\Phi^{*}J-J_{C})|_{g_{C}} +|\nabla^k(\Phi^{*}\omega-\omega_{C})|_{g_{C}}= O(r_{C}^{-\nu-k}),\quad \forall k\in\mathbb N
	\end{equation*}
	where the covariant derivatives are taken with respect to $g_C$. We say that $X$ {\it asymptotic to $C$ with rate $\nu$}. 
	
	\item[(B)]  We say that an open Calabi-Yau manifold $(X, J, \omega, \Omega)$ is {\it asymptotic to the Calabi-Yau cone $(C, J_{C}, \omega_{C}, \Omega_{C})$ with rate $\nu$} if $(X,J,g,\omega)$ is asymptotic to the K\"ahler cone $(C, J_C, g_C, \omega_C)$ with rate $\nu$, and, in addition
	\[
	|\nabla^k(\Phi^{*}\Omega-\Omega_{C})|_{g_{C}}= O(r_{C}^{-\nu-k})
	\]
	\end{enumerate}	
\end{Definition}
\begin{Remark}
\begin{enumerate}
\item 	On any asymptotically conical K\"ahler manifold, we can always find a smooth function $r:X\to \R_{\geq 0}$ satisfying $r = r_C\cdot\Phi^{-1}$ away from some compact set $K$ where $r_C$ is the radial distance on the cone $C$, and furthermore, $r$ satisfies: $|\nabla r|+r|\nabla^2 r|\leq C$. We will call such an $r$ a {\it radius function}. 
\item In fact, it is shown in \cite[Lemma 2.14]{CH13} that $\Phi^{*}J-J_C$ always decays at the same rate as $\Phi^{*}\Omega-\Omega_C$, so it suffices just to assume $|\nabla^k(\Phi^{*}\Omega-\Omega_{C})|_{g_{C}}= O(r^{-\nu-k})$, and $|\nabla^k(\Phi^{*}J-J_C)|_{g_{C}}= O(r^{-\nu-k})$ follows automatically. 
\item We will often say $(X, J, g, \omega)$ is an {\it asymptotically conical K\"ahler manifold} if it is asymptotic to some K\"ahler cone $(C, J_{C}, g_C, \omega_{C})$ at some rate $\nu>0$ by some map $\Phi$. We will therefore often suppress the map $\Phi$, with the understanding that all asymptotics are measured with respect to the diffeomorphism $\Phi$. Furthermore, when $\Phi$ is implicit, we will often abuse notation and write $\omega_C, J_{C}, \Omega_C$ in place of $\Phi^{-1})^{*}\omega_C$, $(\Phi^{-1})^{*}J_C$, $(\Phi^{-1})^{*}\Omega_C$. 
\item On an asymptotically conical K\"ahler manifold with rate $\nu$ we will often refer to a $(1,1)$ form $\alpha$ being asymptotically conical.  By this we mean that there is a compact set $K$ such that, on $X\backslash K$ the form $\alpha$ defines an asymptotically conical K\"ahler metric with rate $\nu$.
\end{enumerate}
\end{Remark}

We now quote two versions of the $\d\db$-lemma which hold on asymptotically conical Calabi-Yau manifolds, see \cite{CH13} for a proof. 

\begin{Proposition}[$\d\db$-lemma, \cite{CH13}, Corollary A.3]\label{ddb-lemma}
	Suppose $X$ is an asymptotically conical K\"ahler manifold with trivial canonical bundle, then 
	\begin{enumerate}
		\item If $\alpha$ is an exact real $(1, 1)$-form on X, then $\alpha = i\d\db u$ for some smooth function $u$.
		\item If $\dim_{\C}X>2$, then if $\alpha$ is an exact real $(1, 1)$-form on $X\setminus K$ for some compact subset $K$, then there exist a compact set $K'$ containing $K$ such that $\alpha = i\d\db u$ on $X\setminus K'$. 
	\end{enumerate}
\end{Proposition}

\begin{Proposition}[Quantitative $\d\db$-lemma, \cite{CH13}, Theorem 3.11]\label{prop: quan-ddbar-lemma}
	Suppose $X$ is an asymptotically conical K\"ahler manifold with $\Ric \geq 0$, then there exist $\eps_0>0$, such that for any $\eta$ an exact $(1, 1)$-form with $\eta\in C^{\infty}_{-\eps}(X)$ for $0<\eps<\eps_0$, then $\eta = i\d\db u$ for $u\in C^{\infty}_{2-\eps}$. 
\end{Proposition}

\subsection{K\"ahler classes on AC K\"ahler manifolds}
We recall the definition of a $\nu$-almost compactly supported class, this is defined in \cite{CH13}, but our definition is slightly different. 

\begin{Definition}\label{decay class}
	Let $X$ be an asymptotically conical K\"ahler manifold, then for any class $[\alpha] \in H^2(X, \R)$, we say that
	\begin{enumerate}
		\item $[\alpha]$ is a K\"ahler class if it contains a positive real $(1, 1)$-form $\alpha>0$
		\item $[\alpha]$ is a $\nu$-almost compactly supported class if it contains a real $(1, 1)$-form $\xi$ satisfying $|\nabla^k \xi| = O(r^{-\nu-k})$
	\end{enumerate}
	and we will denote the set of all $\nu$-almost compactly supported classes by $H^{1, 1}_{\nu}(X)$. 
\end{Definition}

\begin{Remark}
Definition~\ref{decay class} is slightly more restrictive than the definition given in \cite{CH13} where it is only required that the form $\xi$ be defined away from a compact set. But by the second part of Lemma~\ref{ddb-lemma}, the condition in \cite{CH13} implies our condition in the case when $X$ has trivial canonical bundle and $\dim_{\C}X>2$.
\end{Remark}

In \cite{CH13}, it is shown that if $[\alpha]$ is a $\nu$-almost compactly supported and K\"ahler, then one can always construct an asymptotically conical K\"ahler form $\omega\in [\alpha]$ with $|\nabla^k(\omega-\omega_C)| = O(r^{-\nu-k})$. We will recall this construction below in Section~\ref{sec: background}. 

\subsection{Weighted H\"older spaces and solvability of Poisson's equation}
Let us recall some useful Holder spaces defined on asymptotically conical manifolds and some basic theorems regarding the solvability of Poisson equations, which will be useful for us later on. For a detailed treatment of these material, see \cite{LMc, Mar}. 

\begin{Definition}\label{Holder spaces} Let $X$ be a AC K\"ahler manifold as above. 
	\begin{enumerate}
		\item We define the $C^{k, \alpha}_{-\gamma}(X)$ norm of a function as follows
	\begin{equation*}
		\|u\|_{C^{k, \alpha}_{-\gamma}} = \sum_{j=0}^k\sup_X |r^{\gamma+j}\nabla^ju|+[\nabla^k u]_{C^{\alpha}_{-\gamma-k-\alpha}}
	\end{equation*}
	where $r$ is a radius function and
	\begin{equation*}
		[\nabla^k u]_{C^{\alpha}_{-\gamma-k-\alpha}}  = \sup_{x\neq y, d(x, y)\leq \delta}\left[\min(r(x), r(y))^{\gamma+k+\alpha}\frac{|\nabla^k u(x)-\nabla^k u(y)|}{|d(x, y)|^{\alpha}}\right]
	\end{equation*}
	where $\delta>0$ is the convexity radius of $X$, and $|\nabla^ku(x)-\nabla^ku(y)|$ is defined by parallel transporting $\nabla^k u(x)$ along the minimal geodesic from $x$ to $y$. 
	\item We define $C^{\infty}_{-\gamma}(X)$ to be the intersection of $C^{k, \alpha}_{-\gamma}(X)$ over all $k\geq 0$. 
	\item We will also often use the following space $C^{\infty}_{-\gamma}(X\setminus V)$, which we define to be the space of functions $u\in C^{\infty}_{loc}(X\setminus V)$ such that $(1-\chi) u \in C^{\infty}_{-\gamma}(X)$, where $\chi$ is a cutoff function with compact support that is equal to $1$ in a neighborhood of $V$. Where $V$ is the compact analytic subset coming from Assumption~\ref{ass: mainAss}. 
	\end{enumerate}
\end{Definition}
The space $C^{k, \alpha}_{-\gamma}(X)$ consisting of all functions with finite $C^{k, \alpha}_{-\gamma}(X)$ norm is then a Banach space, and it follows that the Laplace operator $\lapl: C^{k+2, \alpha}_{-\gamma+2}(X)\to C^{k, \alpha}_{-\gamma}(X)$ is a bounded map between the two Banach spaces. There is a well-developed Fredholm theory for the Laplace operator on these Banach spaces on an asymptotically conical manifold (see, e.g. \cite{Mar}), which we summarize below. 

\begin{Definition}
	Let $(C, g_C)$ be a Riemannian cone of real dimension $n$ over a smooth compact manifold $L^{n-1}$, then we denote the set of {\em exceptional weights} of the cone $C$, 
	\begin{equation*}
		P = \left\{-\frac{n-2}{2}\pm \sqrt{\frac{(n-2)^2}{4}+\lambda}: \lambda \text{ is an eigenvalue of }\lapl_{L^{n-1}}\right\}. 
	\end{equation*}
	These correspond to the growth rates of homogenous harmonic functions on the cone $(C, g_C)$. 
\end{Definition}
The following theorem summarizes Fredholm theory on an asymptotically conical manifold
\begin{Theorem}[\cite{Mar}, Theorem 6.10]\label{Poisson equation}
	Suppose $(X, g)$ is an asymptotically conical K\"ahler manifold of dimension $2n$. Consider the mapping 
	\begin{equation}\label{laplacian}
	\lapl: C^{k+2, \alpha}_{-\gamma}(X)\to C^{k, \alpha}_{-\gamma-2}(X)
	\end{equation} 
	and let $P$ be the set of exceptional weights of the asymptotic cone $(C, g_C)$.  Then:
	\begin{enumerate}
		\item The operator (\ref{laplacian}) Fredholm if $-\gamma\notin P$. 
		\item The operator (\ref{laplacian}) is surjective if $-\gamma \in (2-2n, \infty)\setminus P$
		\item The operator (\ref{laplacian}) is injective if $-\gamma \in (-\infty, 0)\setminus P$
	\end{enumerate}
\end{Theorem}
\begin{Remark}
	We note that $P\cap (2-2n, 0) = \emptyset$, hence (\ref{laplacian}) is an isomorphism for all $-\gamma\in (2-2n, 0)$. 
\end{Remark}

Now we state a general theorem regarding the solvability of the complex Monge-Amp\`ere equation on an asymptotically conical K\"ahler manifold, which is proved in \cite{CH13}. 

\begin{Theorem}[\cite{CH13}, Theorem 2.4]\label{CH-existence}
    Let $(X, J, \omega)$ be a open K\"ahler manifold asymptotic to a K\"ahler cone $(C, J_C, \omega_C)$ with rate $\nu>0$, and suppose $f\in C^{\infty}_{-\gamma-2}(X)$, then following Complex Monge-Ampere equation then admits a solution
    \begin{equation*}
        (\omega+i\d\db\varphi)^n = e^f\omega^n
    \end{equation*}
    with $\omega_{\varphi} = \omega+i\d\db\varphi >0$ and 
\begin{enumerate}
    \item If $\gamma+2 > 2n$, then we can take $\varphi\in C^{\infty}_{2-2n}$ and $\varphi$ is the unique solution in $C^{\infty}_{2-2n}$. 
    \item If $\gamma+2 \in (2, 2n)$ then we can take $\varphi\in C^{\infty}_{-\gamma}$ and $\varphi$ is the unique solution in $C^{\infty}_{-\gamma}$. 
    \item If $\gamma+2 \in (0, 2)$ and $-\gamma$ is not an exceptional weight, we can take $\varphi\in C^{\infty}_{-\gamma}$. 
\end{enumerate}
\end{Theorem}

\subsection{Proof of Theorem \ref{Main Theorem}}
We breakdown the proof of Theorem~\ref{Main Theorem} in the following two propositions, and we will give the proof of Theorem~\ref{Main Theorem} assuming these results.   We will prove Proposition~\ref{Background metrics} in Section~\ref{sec: background} and Proposition~\ref{estimates} in Section~\ref{sec: estimates}. Theorem~\ref{Theorem2} will be proved in section~\ref{sec: metricGeometry}. 
\begin{Proposition}[Constructing background metrics]\label{Background metrics}
	Suppose $\nu>0$, and let $(X, J, \omega, \Omega)$ be an asymptotic to a Calabi-Yau cone $(C, J_C, \omega_C, \Omega_C)$ with rate $\nu$. Suppose that $-\nu \in (-2n, 0)$ and $-\nu+2$ is not an exceptional weight. Suppose $[\alpha_t] = (1-t)[\alpha_0]+t[\alpha_1]\in H^{1, 1}_{\nu}(X)$ is a linear family of K\"ahler classes in $H^{1,1}_{\nu}$ for $t\in (0, 1]$, and suppose that $[\alpha_0] \in H^{1,1}_{\nu}$ has a semi-positive representative $\alpha_0$.  Then there exists $\epsilon >0$, a compact set $K\subset X$ and a smooth family of real $(1, 1)$-forms $\hat{\omega}_t\in [\alpha_t]$ for $t\in [0, \eps]$ satisfying the following:
	\begin{enumerate}
		\item $\hat{\omega}_t >0$ for all $t\in (0, \epsilon]$.  
		\item $\hat{\omega}_0 \geq 0$ and $\hat{\omega}_0 = \alpha_0$ on a compact set $K\subset \subset X$.  (In fact, we can choose this compact set $K$ to be as large as we like)
		\item On $X\backslash K$ there holds  $|\nabla^k(\hat{\omega}_t-\omega_C)|_{g_C} \leq Cr^{-\nu-k}$ for all $t\in [0, \epsilon]$ for a constant $C$ independent of $t$. 
		\item There exist $\gamma>0$ such that, on $X\backslash K$ the Ricci potentials $f_t = \log \frac{i^{n^2}\Omega\wedge\bar{\Omega}}{\hat{\omega}_t^n}$ satisfy the asymptotics $|\nabla^kf_t|\leq Cr^{-\gamma-2-k}$ uniformly in $t$. 
	\end{enumerate}	
\end{Proposition}

\begin{Proposition}[A priori estimates]\label{estimates}
	Let $(X, J, \omega, \Omega)$ be asymptotic to a Calabi-Yau cone $(C, J_C, \omega_C, \Omega_C)$ with rate $\nu>0$, and $ H^{1, 1}_{\nu}(X)\ni [\alpha_t] = (1-t)[\alpha_0]+t[\alpha_1]$ is a linear family of K\"ahler classes for $t\in (0, 1]$ satisfying Assumption \ref{ass: mainAss}, and let $\hat{\omega}_t\in [\alpha_t]$ be the forms constructed in Proposition~\ref{Background metrics}. Let $\varphi_t$ be the solution of the complex Monge-Amp\`ere equations
	\begin{equation}\label{familycxma}
		(\hat{\omega}_t+i\d\db\varphi_t)^n = e^{f_t}\hat{\omega}_t^n (= i^{n^2}\Omega\wedge\bar{\Omega})
	\end{equation}
	obtained from Theorem \ref{CH-existence}. Then the following estimates hold uniformly in $t$
	\begin{enumerate}
		\item $|\varphi_t|\leq C$. 
		\item $\varphi_t$ is uniformly bounded in ${C^{\infty}_{loc}(X\setminus V)}$. 
		\item There exist a compact subset $K\subset X$ containing $V$ such that the following estimate hold outside of $K$
		\begin{equation*}
			|\nabla^k\varphi_t|\leq Cr^{-\gamma-k}
		\end{equation*}
		for $C$ independent of $t$. 
	\end{enumerate}
\end{Proposition}

Now we prove Theorem \ref{Main Theorem} given the above two propositions
\begin{proof}[Proof of Theorem \ref{Main Theorem}]
	Let $[\alpha_t] = (1-t)[\alpha]+t[\eps \omega]$, then by Proposition \ref{Background metrics}, we can construct a sequence of background metrics $\hat{\omega}_t\in [\alpha_t]$ satisfying the properties stated in the Proposition. Then using these as background metrics, we can write down a family of complex Monge-Ampere equations
	\begin{equation*}
		(\hat{\omega}_t+i\d\db\varphi_t)^n = e^{f_t}\hat{\omega}_t^n (= i^{n^2}\Omega\wedge\bar{\Omega})
	\end{equation*}
	then by the Theorem \ref{CH-existence}, the equations are solvable for $t>0$, and Proposition \ref{estimates} applies to the family of solutions $\varphi_t$. Once we have the a priori estimate, it's then clear that by taking a subsequence, we can take a limit $\varphi_{t_i}\to\varphi_0$ in $C^{\infty}_{loc}(X\setminus V)$, which satisfies the equation
	\begin{equation}\label{CXMA}
		(\hat{\omega}_0+i\d\db\varphi_0)^n = i^{n^2}\Omega\wedge\bar{\Omega}
	\end{equation}
	smoothly away from the analytic set $V$. Moreover, $\varphi_0$ is a bounded by the uniform $C^0$ estimate of $\varphi_t$, hence $\hat{\omega}_0+i\d\db\varphi_0$ extends as a non-negative current on $X$ by \cite{GR}, and it does not charge any analytic subsets, so the equation \eqref{CXMA} holds globally.  From Proposition~\ref{Background metrics} (2), and Proposition~\ref{estimates} (3), we see that $\omega_{\varphi_0}$ is asymptotically conical.  It only remains to establish the incompleteness and uniqueness statements of $\omega_{\varphi_0}$ in Theorem~\ref{Main Theorem}.  The incompleteness of $\omega_{\varphi_0}$ follows from the diameter bound in Lemma~\ref{diam-bdd}, while the uniqueness is established in Theorem~\ref{thm: uniqueness}
\end{proof}

\section{Background metrics}\label{sec: background}
The goal of this section is to prove Proposition \ref{Background metrics}, which constructs a family of ``good" background metrics $\hat{\omega}_t\in [\alpha_t]$ whose Ricci potentials decay faster than quadratically. Indeed, it is easy to construct $\omega_t\in [\alpha_t]$ satisfying only the first two conditions of Proposition~\ref{Background metrics}.  However, the proof of the a priori estimates of Proposition~\ref{estimates} depends crucially on the additional decay of the Ricci potentials. This idea is used in \cite{CH13} (see also \cite[Prop. 4.2.6]{Co}). 
		
From now on we fix an open Calabi-Yau manifold $(X, J, \Omega)$ asymptotic to some Calabi-Yau cone $(C, J_C, \Omega_C, \omega_C, g_{C})$ at rate $\nu>0$. In the following proposition, we summarize a construction of asymptotically conical K\"ahler (semipositive) forms in almost compactly support classes, which is based on \cite{CH13}. 
\begin{Proposition}\label{bkground-metric 1}
	Suppose $[\alpha]\in H^{1, 1}_{\nu}(X)$ contains a (semi-)positive form $\alpha$, then there exist a (semi-)positive form $\omega \in [\alpha]$ which agrees with $\alpha$ in a compact set $K$ and satisfies the asymptotics $|\nabla^k (\omega-\omega_C)| = O(r^{-\nu-k})$ for $r \gg 1$. 
\end{Proposition}
\begin{proof}
	This follows from construction in \cite[Theorem 2.4]{CH13}.

\end{proof}

\begin{Proposition}\label{prop: bkground-metric-improvement}
    Suppose that $(X, J, \Omega,  \omega_t, g_t)_{t\in [0, 1]}$ are a smooth family of data which is asymptotic to the cone $(C, J_C, \Omega_C, \omega_C, g_C)$ at the rate $-\nu \in (-2, 0)$. Suppose that for $t\in (0,1]$, $\omega_t$ are asymptotically conical K\"ahler metrics and $\omega_0$ is asymptotically conical and semi-positive $(1, 1)$ form. Let $f_t$, $t\in[0,1]$ be the Ricci potentials of $\omega_t$, defined by $e^{f_t} = \frac{i^{n^2}\Omega\wedge\bar{\Omega} }{\omega_t^n}$, and suppose there is a compact set $K\subset X$ so that on $X\backslash K$, $f_t$ satisfy the following asymptotics:
    \begin{enumerate}
    	\item $|f_t|\leq Cr^{-\beta}$
    	\item $|\nabla^k f_t|_{g_C}\leq Cr^{-\beta-k}$
    \end{enumerate}
	where $C$ is independent of $t$ and $\nu\leq \beta <2n-2$ and $-\beta+2$ is not an exceptional weight. 
	
    Then there exist $\eps>0$ and a family of functions $u_t$ for $t\in [0, \eps]$ such that the following are satisfied
	\begin{enumerate}
		\item There exist a compact subset $K\subset X$ such that $\text{supp}(u_t)\subset X\setminus K$
		\item $\omega_t+i\d\db u_t > 0$ on $\text{supp}(u_t)$
		\item $|\nabla^k u_t|_{g_C}\leq Cr^{-\beta+2-k}$
		\item $|\nabla^k \frac{\d u_t}{\d t}|_{g_C}\leq Cr^{-\beta+2-k}$
		\item Away from a compact set $K$, we have
		\[(\omega_t+i\d\db u_t)^n = e^{f_t-f^'_t}\omega_t^n = e^{-f^'_t}i^{n^2}\Omega\wedge\bar{\Omega}\]
		where $|\nabla^k f^'_t|\leq Cr^{-2\beta-k}$ outside a compact set $K$. 
		
	\end{enumerate}
where the constant $C$ is independent of $t$. In particular, this means if we set $\omega_t'  = \omega_t+i\d\db u_t$, then $\omega'_t$ converges to $\omega_C$ at the same rate as $\omega_t$, but the Ricci potentials $f_t'$ of $\omega_t'$ decays a rate of $-2\beta$.
\end{Proposition}

\begin{proof}
    We can essentially follow the same procedure as in \cite[Lemma 2.12]{CH13}.  First we want to solve the equation
    \begin{equation*}
        \lapl_{\omega_t}\hat{u}_t = 2f_t
    \end{equation*}
    for $t\geq 0$, away from a compact set while controlling of the growth of the solutions.
    
      We now fix a standard cutoff function $\chi:\R\to \R$ with
    \begin{equation*}
    \chi(x) = 
    \begin{cases}
    0       & \quad \text{for } x \leq 1\\
    1  & \quad \text{for } x \geq 2
    \end{cases}
    \end{equation*}
    and satisfy $0\leq\chi\leq 1$, $|\chi'|\leq 2$, $|\chi''|\leq 5$. Then we define $\zeta_R:X\to \R$ by setting $\zeta_R(x) = \chi(\frac{r(x)}{R})$, and let $\hat{g}$ be any metric on $X$. Then set
    \[
    \bar{g}_t = (1-\zeta_R)\hat{g}+g_t
    \]
    Since $\omega_0$ is semi-positive and asymptotically conical we can choose $R$ sufficiently large so that $\bar{g}_0$ is an asymptotically conical Riemannian metric.  Then for all $t \in [0,1]$,  $g_t$ defines a background metric and for $t\in (0,1]$, this metric is equal to the $\omega_t$  away from a compact set.

    If $-\beta+2$ is not an exceptional weight, then $\lapl_{\bar{g}_t}:C^{\infty}_{-\beta+2}\to C^{\infty}_{-\beta}$ is surjective by Theorem~\ref{Poisson equation}, so we can always solve the equation
    \begin{equation*}
        \lapl_{\bar{g}_t}\hat{u}_t = 2\zeta_Rf_t
    \end{equation*}
    for $\hat{u}_t\in C^{\infty}_{-\beta+2}$. In fact, by the Implicit Function Theorem \cite[Proposition 4.2.19]{DKr}, we can find a family of smoothly varying solutions for $t\in [0, \eps)$, and such that following bounds hold uniformly for small $t$. 

    \begin{enumerate}
    	\item $|\nabla^k \hat{u}_t|\leq Cr^{-\beta+2-k}$
    	\item $|\nabla^k \frac{\d \hat{u}_t}{\d t}|\leq Cr^{-\beta+2-k}$
    \end{enumerate} 
    If we set $u_t = \zeta_S\hat{u}_t$ then $u_t$ is supported on $\text{supp}(\zeta_S)$, and then we have
    \begin{align*}
    	|i\d\db u_t|&\leq
    	|\zeta_S||\d\db\hat{u_t}|+ |\hat{u_t}||\d\db\zeta_S|+2|\d\zeta_S||\d u_t|\\
    	&\leq  C\zeta_Sr^{-\beta}+Cr^{-\beta+2}|\d\db\zeta_S|+Cr^{-\beta+1}|\nabla \zeta_S|\\
    	&\leq Cr^{-\beta}(\zeta_S+r|\nabla\zeta_S|+r^2|i\d\db \zeta_S|)\\
    	&\leq Cr^{-\beta}(\zeta_S+S|\nabla\zeta_S|+S^2|i\d\db \zeta_S|)
    \end{align*}
    but since $\zeta_S(x) = \chi(\frac{r(x)}{S})$, we see that 
    \[|\nabla \zeta_S| = S^{-1}|\chi' \nabla r|\leq C S^{-1}\] and 
    \[|i\d\db \zeta_S| \leq S^{-2}|\chi''||\nabla r|^2+|\chi'|S^{-1}|i\d\db r| \leq CS^{-2}\]
    where we used that $r|i\d\db r|\leq C$. 
    So we have
    \begin{equation*}
    	|i\d\db u_t|\leq Cr^{-\beta}(\zeta_S+C)
    \end{equation*}
    and $i\d\db u_t$ is supported on the support of $\zeta_S$. Hence for $S$ sufficiently large, we can ensure that $\omega_t+i\d\db u_t>0$ on the $\text{supp}(u_t)$.  
	
	Away from the compact set $K$, we have
	\begin{align*}
		\frac{(\omega_t+i\d\db u_t)^n}{\omega_t^n} &= 1 + f_t + O(|i\d\db u_t|^2)\\
		& = 1+f_t+O(r^{-2\beta})
	\end{align*}
	so setting $f'_t = f_t - \log\frac{(\omega_t+i\d\db u_t)^n}{\omega_t^n}$, we have
	\begin{equation*}
		(\omega_t+i\d\db u_t)^n = e^{f_t-f'_t}\omega_t^n
	\end{equation*}
	and $f'_t = f_t - \log (1+f_t+O(r^{-2\beta}))$ has the desired asymptotics. 
\end{proof}

\begin{Remark}
	If $-\beta+2$ is an exceptional weight, we can apply the proposition with $\beta+\eps$ in place of $\beta$ for $\eps$ arbitrarily small (since the exceptional weights are discrete). We can then repeatedly apply Proposition~\ref{prop: bkground-metric-improvement} to improve the decay of Ricci potential for a family of metrics until we obtain the decays we need. 
\end{Remark}

The two previous propositions combined proves Proposition \ref{Background metrics}. 

\begin{proof}[Proof of Proposition \ref{Background metrics}]
	By Proposition 3.1, we can find a semi-positive form $\omega_0\in [\alpha_0]$ satisfying the asymptotics $|\nabla^k(\omega_C-\omega_0)| = O(r^{-\nu-k})$ and a metric $\omega_{1}\in [\alpha_{1}]$ satisfying the same asymptotics, then if we write $\omega_t$ by linearly interpolating between $\omega_0$ and $\omega_{1}$, then clearly $\omega_t$ are positive for $t>0$ and satisfy the desired asymptotics, and the Ricci potentials $f_t$ satisfy $|\nabla^kf_t|\leq C(1+r)^{-\nu-k}$. If $\nu>2$, then we can take $\gamma = \nu$ and we are done, otherwise, we can apply Proposition~\ref{prop: bkground-metric-improvement} repeatedly to improve the asymptotics of the Ricci potentials until they decay faster than quadratically. 
\end{proof}

\subsection{K\"ahler currents and Null loci in the asymptotically conical case}\label{sec: Null=Enk}

Before proceeding we would like to briefly discuss Assumption~\ref{ass: mainAss}.  Recall that if $(X,\omega)$ is compact K\"ahler and $[\alpha] \in \overline{K}$ is a nef class with $\int_{X}\alpha^n >0$, then, by results of Demailly-P\u{a}un \cite{DP} there is a function $\psi : X \rightarrow \mathbb{R}\cup\{-\infty\}$ such that
\[
 \alpha + \sqrt{-1}\d\db \psi \geq \epsilon \omega
\]
for some $\epsilon >0$, $\psi$ is smooth on the complement of an analytic subset $Z$, and $\{\psi =- \infty\} =Z$.  Furthermore, by results of the first author and Tosatti \cite{CT} $\psi$ can be chosen so that the analytic subvariety $Z$ is given by
\[
\text{Null}(\alpha) := \bigcup_{\int_{V}\alpha^{\dim V} = 0}V
\]
where  the union is taken over irreducible analyitic subvarieties $V\subset X$.  We expect that a similar result holds in the asymptotically conical setting.  We make the following conjecture

\begin{Conjecture}\label{conj: EnK=Null}
Suppose $[\alpha] \in H^{1,1}_{\nu}(X,\mathbb{R})$ is a limit of $\nu$-almost compactly support K\"ahler classes.  Then there is a function $\psi: X \rightarrow \mathbb{R}\cup\{-\infty\}$ such that $\alpha +\sqrt{-1}\d\db \psi \geq \epsilon \omega$ for some asymptotically conical K\"ahler form $\omega$. Define 
\begin{equation}\label{eq: nullAC}
\text{Null}(\alpha) := \bigcup_{\int_{V}\alpha^{\dim V} = 0}V
\end{equation}
where the union is taken over all compact, irreducible, analytic subvarieties $V\subset X$.  Then $\text{Null}(\alpha)$ is an analytic subvariety, and $\psi$ can be chosen so that $\psi$ is smooth on $X\backslash \text{Null}(\alpha)$ and 
\[
\{\psi = -\infty\} = \text{Null}(\alpha). 
\]
\end{Conjecture}

At a purely moral level, the reason that non-compact analytic subvarieties should not enter into the definition of $\text{Null}(\alpha)$ in the asymptotically conical setting is that, at least when $[\alpha]$ admits a semi-positive representative, Proposition~\ref{bkground-metric 1} yields the existence of a form $\hat{\alpha} \in [\alpha]$ which is asymptotically conical.  Thus, if $V$ is a non-compact subvariety, then $\int_{V}\hat{\alpha}^{\dim V} = +\infty$.  Of course, this is purely moral reasoning, since the integral $\int_{V}\hat{\alpha}^{\dim V}$ is not independent of the representative of $[\alpha]$.

\begin{Lemma}\label{lem: DP}
Conjecture~\ref{conj: EnK=Null} holds when, $[\alpha]$ is semi-positive and the cone at infinity is quasi-regular.
\end{Lemma}

Recall that the cone $(C, J_C, \Omega_C, \omega_C, g_C)$ is quasi-regular if the holomorphic vector field $r_C\frac{\d}{\d r_{C}} - \sqrt{-1}J_{C}\left(r_C\frac{\d}{\d r_{C}} \right)$ integrates to define a $\mathbb{C}^*$ action.
\begin{proof}
By a result of Conlon-Hein \cite{CH3}, building on work of Li \cite{Li}, if $(X,J,\Omega, \omega, g)$ is asymptotically conical Calabi-Yau with quasi-regular Calabi-Yau cone at infinity, then there is a complex, projective orbifold $M$ without codimension $1$ singularities, and a orbidivisor $D$ with positive normal orbibundle such that $M= X\cup D$, and $-K_{M} = q[D]$ for some $q \geq 1$.  Furthermore, every K\"ahler form on $X$ is cohomologous to the restriction of a K\"ahler form on $M$, and the restriction map $H^{1,1}(M)\rightarrow H^{2}(X)$ is surjective.  Let $[\omega_t] = (1-t)[\alpha_0]+ t[\omega_0] \in H^{1,1}(X)$ be a family of $\nu$-almost compactly supported K\"ahler classes for $t\in (0,1]$ such that $[\alpha_0]$ is semi-positive.  In fact, according to \cite[Proposition  2.5]{CH13} all K\"ahler classes on $X$ are $2$-almost compactly supported, so the assumption of almost compact support can be dropped.  Let $[\hat{\omega}], [\hat{\alpha}_0] \in H^{1,1}(M)$ be such that $[\hat{\omega}]$ is K\"ahler, and $[\hat{\omega}]\big|_{X} = [\omega_0], [\hat{\alpha}_0]\big|_{X}= [\alpha_0]$. Since $\alpha_0$ is semi-positive, and $D$ has positive normal bundle, the argument in the proof of \cite[Theorem A]{CH3} shows that we can find a constant $C>0$ so that $[\hat{\alpha}_0] + C[D]$ is semi-positive, and positive in a neighborhood of $D$. Furthermore, since $D|_{D}$ is positive, after possibly increasing $C$ we can assume that
\[
\int_{M}([\hat{\alpha}_0] + C[D])^n >0
\]
Let $\pi:\overline{M}\rightarrow M$ be a resolution of singularities, obtained by blowing up smooth centers.  Since $X$ is smooth, and $M$ has only codimension $2$ singularities, we can assume that $\pi|_{X}$ is an isomorphism, and that $\pi$ is an isomorphism at the generic point of $D$. Let $E$ denote the exceptional divisor of $\pi$, and let $\bar{D}= \pi^{-1}(D)$ be the total transform of $D$. Now we have
\[
\pi^*[\hat{\alpha}_0] +C[\bar{D}] 
\]
is nef, and big by Demailly-P\u{a}un \cite{DP}.  By the results of \cite{DP} and the first author and Tosatti \cite{CT} there is a K\"ahler current in $\pi^*[\hat{\alpha}_0] +C[\bar{D}]$ which is smooth on the complement of ${\rm Null}(\pi^*[\hat{\alpha}_0] +C[\bar{D}])$.  Let $Y\subset \bar{M}$ be an irreducible analytic subvariety of dimension $p>0$.  If $Y\cap \pi^{-1}(D) = \emptyset$, then
\[
\int_{Y}(\pi^*[\hat{\alpha}_0] +C[\bar{D}])^p = \int_{\pi(Y)} \alpha_0^{p}
\]
and so $Y\subset {\rm Null}(\pi^*[\hat{\alpha}_0] +C[\bar{D}])$ if and only if $\pi(Y) \subset {\rm Null}([\alpha_0])$.  Now suppose that $Y \cap \pi^{-1}(D) \cap \pi^{-1}(X) \ne \emptyset$.  Let $\hat{\alpha}_0+ C\beta_{D}+\sqrt{-1}\d\db u$ be the smooth semi-positive representative of $[\hat{\alpha}_0] +C[D]$ which is positive in a neighborhood of $D$.  Then, since $\pi$ is an isomorphism at the generic point of $Y$ we have
\[
\int_{Y}(\pi^*[\hat{\alpha}_0] +C[\bar{D}])^p =\int_{Y\backslash (E\cap Y)}[\pi^* (\hat{\alpha}_0+ C\beta_{D}+\sqrt{-1}\d\db u)]^p \int_{\pi(Y)}(\hat{\alpha}_0+ C\beta_{D}+\sqrt{-1}\d\db u)^p >0,
\]
where the last inequality follows from the fact that $\hat{\alpha}_0+ C\beta_{D}+\sqrt{-1}\d\db u \geq 0$ and there is a neighborhood of $\pi(Y)\cap D$ where $\hat{\alpha}_0+ C\beta_{D}+\sqrt{-1}\d\db u > 0$.  Thus we have
\[
{\rm Null}(\pi^*[\hat{\alpha}_0] +C[\bar{D}]) \cap (\pi^{-1}(D))^c = \pi^{-1}\left({\rm Null}([\alpha_0])\right).
\]
Since $\pi: \bar{M}\backslash \pi^{-1}(D) \rightarrow X$ is an isomorphism, the result follows.

\end{proof}

\section{A priori estimates}\label{sec: estimates}
In this section, we prove Proposition \ref{estimates}.  Let us first recall the general setup of the proposition. Let $(X, J, \omega, \Omega)$ be an asymptotically conical Calabi-Yau manifold which is asymptotic to the Calabi-Yau cone $(C, J_C, \omega_C, \Omega_C)$ with rate $\nu>0$, and $[\alpha_t] = (1-t)[\alpha_0]+t[\alpha_1] \in H^{1, 1}_{\nu}$ for $t\in [0, 1]$ is a family of $\nu$-almost compactly supported classes such that $[\alpha_t]$ is K\"ahler for $t>0$. Suppose $[\alpha_0]$ satisfies Assumption~\ref{ass: mainAss}. Then let $\hat{\omega}_t\in [\alpha_t]$ for $t\in (0, 1]$ be a family of asymptotically conical K\"ahler metrics satisfying the conclusion of Proposition~\ref{Background metrics}. Then by Theorem~\ref{CH-existence}, we can solve the equation
\begin{equation*}
	(\hat{\omega}_t+i\d\db\varphi_t)^n  = i^{n^2}\Omega\wedge\bar{\Omega} (=e^{f_t}\hat{\omega}_{t}^n)
\end{equation*}
for $\varphi_t\in C^{\infty}_{-\gamma}(X)$, the our goal in this section is to prove a priori estimates on the potentials $\varphi_t$ that are uniform in $t$ as $t\to 0$. 

\subsection{Uniform estimates}\label{subsec: C0bd}
In this section, we prove a uniform bound for $\varphi_t$ that is independent of $t$. In the compact case, such an estimate can be proved using pluripotential theory following the seminal work of Kolodziej \cite{Ko}, see \cite{EGZ07}. Pluripotential methods allow one to obtain an estimate with a sharper dependence on the data of the right hand side. However, such methods are hard to adapt to the non-compact setting and no proper analogue of such estimates are known. It would be of interest to try to find extensions of the pluripotential estimates to the non-compact setting, as it would give a sharper estimates which would apply more generally to singular Calabi-Yau manifolds not admitting crepant resolutions. 

Instead, we will use an idea based on the original argument of Yau \cite{Yau78} using the Moser iteration.  However, following an idea of Tosatti \cite{To07} we perform the Moser iteration using the Calabi-Yau metrics $\omega_{\varphi_t} : =\hat{\omega}_t+i\d\db\varphi_t$ as background metrics.  The advantage of this trick is that since the metrics $\omega_{\varphi_t}$ are Ricci flat and asymptotically conical, they have a uniform Sobolev inequality by results of Croke \cite{Cr} and Yau \cite{YauSurv}. 
\begin{Proposition}
    The metrics $\omega_{\varphi_t}$ satisfy a uniform Sobolev inequality of the form
    \begin{equation}\label{eq: Sobolev}
        \left(\int_X|u|^{\frac{2n}{n-1}}i^{n^2}\Omega\wedge\bar{\Omega}\right)^{\frac{n-1}{n}}\leq C\int_{X}|du|_{\omega_{\varphi_t}}^2 i^{n^2}\Omega\wedge\bar{\Omega}
    \end{equation}
\end{Proposition}
\begin{proof}
It suffices to prove the result for compact supported smooth functions.  Results of Croke \cite{Cr} and Yau \cite{YauSurv} show that for a compactly supported function $u$, with $\text{supp}(u) \subset \Omega$ for an arbitrary relatively compact set $\Omega\subset X$,~\eqref{eq: Sobolev} holds for a constant $C$, depending on an upper bound for the diameter of $\Omega$, a lower bound for the volume of $\Omega$, and a lower bound for the Ricci curvature.  We only need to exploit the scale invariance of these quantities for asymptotically conical Calabi-Yau metrics.     Fix a point $x_0 \in X$.  Since $\omega_{\phi_t}$ are asymptotically conical, for $R$ sufficiently large we have
\[
{\rm Vol}_{\omega_{\phi_t}}(B_{R}(x_0)) \sim R^{2n}{\rm Vol}_{\omega_{C}}(L)
\]
where $L$ is the link of the cone, identified with $\{r_{C}=1\} \subset C$, and the volume is computed using the conical Calabi-Yau metric $\omega_{C}$.  Therefore, if $\omega_{R} = R^{-2}\omega_{\phi_{t}}$, then with respect to the rescaled metric the diameter is $1$, and the volume is ${\rm Vol}_{\omega_{C}}(L)$.  Since~\eqref{eq: Sobolev} is scale invariant, the result follows.
\end{proof}
\begin{Proposition}\label{prop: Moser It C0}
    Given solutions $\varphi_t$ to \eqref{familycxma}, with $|\nabla^k\varphi| = O(r^{-\gamma-k})$ we have the following uniform estimate for the potential 
    \begin{equation*}
        |\varphi_t|\leq C\|\varphi_t\|_{L^p(i^{n^2}\Omega\wedge\bar{\Omega})}
    \end{equation*}
    for any $p>\frac{2n-2}{\gamma}\geq 1$ and $C$ depending on $n$, $p$, and a uniform bound on $\|e^{-f_t}-1\|_{L^q}$ for $q\in [p, \infty]$. 
\end{Proposition}
\begin{proof}
If we set $T_t = \sum_{k=0}^{n-1}\omega_{\varphi_t}^k\wedge\hat{\omega}_t^{n-1-k}$, then we can rewrite the equation as 
\begin{equation*}
    -i\d\db\varphi_t \wedge T_t = (e^{-f_t}-1)i^{n^2}\Omega\wedge\bar{\Omega}
\end{equation*}
multiplying both sides by $|\varphi_t|^{p-2}\varphi_t$ and integrating, we get

\begin{equation*}
        -\int_M |\varphi_t|^{p-2}\varphi_t i\d\db\varphi_t \wedge T_t = \int_M |\varphi_t|^{p-2}\varphi_t (e^{-f_t}-1)i^{n^2}\Omega\wedge\bar{\Omega}
\end{equation*}
we will integrate by parts on the first term
\begin{align*}
    -\int_M |\varphi_t|^{p-2}\varphi_t i\d\db\varphi_t \wedge T_t & = \lim_{R\to \infty}\left(-\int_{B_R} |\varphi_t|^{p-2}\varphi_t i\d\db\varphi_t \wedge T_t\right)\\
    & = \lim_{R\to \infty}\left((p-1)\int_{B_r}|\varphi_t|^{p-2}i\d\varphi_t\wedge\db\varphi_t\wedge T_t -\int_{\d B_R}|\varphi_t|^{p-2}\varphi_t i\db\varphi_t \wedge T_t \right)\\
    & = \frac{4(p-1)}{p^2}\int_{M}i\d|\varphi_t|^{\frac{p}{2}}\wedge\db|\varphi_t|^{\frac{p}{2}}\wedge T_t - \underbrace{\lim_{R\to \infty}\int_{\d B_R}|\varphi_t|^{p-2}\varphi_t i\db\varphi_t \wedge T_t}_{=0 \text{ for } p > \frac{2n-2}{\gamma}}
\end{align*}
Combined with the Sobolev inequality, we have
\begin{equation*}
    \left(\int_M |\varphi_t|^{p\frac{n}{n-1}}i^{n^2}\Omega\wedge\bar{\Omega}\right)^{\frac{n-1}{n}}\leq C\frac{np^2}{4(p-1)}\int_M |\varphi_t|^{p-1} |e^{-f_t}-1|i^{n^2}\Omega\wedge\bar{\Omega}
\end{equation*}
for any $p>\frac{2n-2}{\gamma}$. By H\"older's inequality, we have (below $\frac 1 q + \frac{1}{q'} = 1$)
\begin{equation}\label{iteration}
    \|\varphi_t\|_{L^{p\frac{n}{n-1}}}^p\leq C\frac{np^2}{4(p-1)}\||\varphi_t|^{p-1}\|_{L^{q}}\|e^{-f_t}-1\|_{L^{q'}} = C\frac{np^2}{4(p-1)}\|\varphi_t\|_{L^{q(p-1)}}^{p-1}\|e^{-f_t}-1\|_{L^{q'}}
\end{equation}
picking $q$ such that $q = \frac{p}{p-1}>1$, we get
\begin{align*}
    \|\varphi_t\|^p_{L^{p\frac{n}{n-1}}}&\leq \frac{C_Snp^2}{4(p-1)}\|\varphi_t\|^{p-1}_{L^p}\|e^{-f_t}-1\|_{L^{p}}\\
    &\leq \frac{CC_Snp^2}{4(p-1)}\|\varphi_t\|^{p-1}_{L^p}
\end{align*}
a standard Moser iteration argument gives the result. 

\end{proof}

\begin{Proposition}\label{lp-bound}
    For any $p>\frac{2n}{\gamma}$, we have a uniform $L^p$ estimate of the form
    \begin{equation*}
        \|\varphi_t\|_{L^p}\leq C
    \end{equation*}
    for $C$ depending on $n$, $p$ and $\|e^{-f_t}-1\|_{L^{\frac{np}{n+p}}}$. 
\end{Proposition}

\begin{proof}
In equation \eqref{iteration}, if we pick $q>1$, such that $q(k-1) = k\frac{n}{n-1}$, we get
\begin{equation*}
    \|\varphi_t\|_{L^{k\frac{n}{n-1}}}\leq C\frac{np^2}{4(p-1)}\|e^{-f_t}-1\|_{L^{\frac{nk}{n+k-1}}}
\end{equation*}
taking $p = \frac{nk}{n-1}$ gives us our result. 
\end{proof}

\begin{Corollary}\label{cor: C0-bdd}
	The potentials $\varphi_t$ are bounded in $L^p$ uniformly in $t$ for any $p\in (\frac{2n}{\gamma}, \infty]$, 
	\begin{equation*}
		\|\varphi_t\|_{L^p}\leq C_p
	\end{equation*}
	In particular, th potentials $\varphi_t$ are uniformly bounded in $C^0$. 
\end{Corollary}
\begin{proof}
	This follows by combining Proposition~\ref{prop: Moser It C0} and Proposition~\ref{lp-bound}. Note that since $|f_t|\leq Cr^{-\gamma-2}$ outside a fixed compact set, we have an estimate $\|e^{-f_t}-1\|_{L^{\frac{np}{n+p}}}\leq C$ for a constant $C$ independent of $p,t$ for any $p>\frac{2n-2}{\gamma}$.
\end{proof}

\subsection{Convergence of the metric away from the degeneracy locus}
In this section, we prove an estimate for $\d\db\varphi_t$ away from $V$, the subvariety coming from Assumption~\ref{ass: mainAss}. Recall that by Assumption~\ref{ass: mainAss}, there exist $\psi\in PSH(X, \alpha_0)$ which is smooth outside of $V$ and goes to $-\infty$ near $V$, the idea is to use this function as a barrier function in the $C^2$ estimate, and this is first used by Tsuji in in \cite{Tsuji} to study K\"ahler-Ricci flow. We remark that this is the only part of the Theorem that uses the current in Assumption~\ref{ass: mainAss}. 

Before we prove the estimate, we first construct a slightly more better behaved barrier function $\psi_{\eps}\in PSH(X, \hat{\omega}_0)$ which is compactly supported. Recall that from the construction of $\hat{\omega}_0$, $\hat{\omega}_0$ is equal to $\alpha_0$ on a large compact set. (which from the construction can be as large as one want) 
\begin{Lemma}\label{barrier}
	 There exist $\psi_{\eps}\in PSH(X, \hat{\omega}_0)$ which is compactly supported and satisfy $\hat{\omega}_0+i\d\db\psi_{\eps}\geq \eps \omega$, and is smooth outside $V$ and goes to $-\infty$ near $V$. 
\end{Lemma}
\begin{proof}
	Recall by \cite[Lemma 2.15]{CH13}, we know that $r^{2\kappa}$ for $\kappa \in (0, 1)$ is strictly plurisubharmonic for $r$ sufficiently large, and satisfies
	\[|\nabla r^{2\kappa}| = O(r^{2\kappa-1}) \qquad|i\d\db r^{2\kappa}| = O(r^{2\kappa-2})\]
	Pick $\Psi: \R^{+}\to \R^{+}$ smooth satisfy $\Psi', \Psi'' \geq 0$  and
	 \begin{equation*}
	\Psi(x) = 
	\begin{cases}
	T+2       & \quad \text{for } x < T+1\\
	x  & \quad \text{for } x > T+3
	\end{cases}
	\end{equation*}
	then as in \cite[Lemma 2.15]{CH13}, for $T\gg 1$, $\Psi(r^{2\kappa})$ is plurisubharmonic and equal to $r^{2\kappa}$ for $r$ sufficiently large. 
	
	We set 
	\[\psi_{\eps} = (1-\zeta_S)\psi+ C(1-\zeta_R)\Psi(r^{2\kappa})\]
	where $S, C, R$ are chosen as follows. First we pick $S\gg 1$ large enough such that $\hat{\omega}_{0} = \alpha_0$ on $\{r\leq S\}$ and $i\d\db\Psi(r^{2\kappa})> 0$ on $\{S\leq r\leq 2S\}$, which implies that $\hat{\omega}_0+i\d\db \psi_{\eps} = \alpha_0+i\d\db\psi\geq \eps_0\omega$ on $\{r\leq S\}$. Then pick $C\gg 1$ large enough so that $Ci\d\db\Psi(r^{2\kappa})> i\d\db ((1-\zeta_S)\psi)$ on $\{S\leq r\leq 2S\}$. Finally, we pick $R\gg S$ such that $\hat{\omega}_0+i\d\db \psi_{\eps}>0$ on $\{R\leq r\leq 2R\}$, which is possible since for $R$ large, we have
	\[|i\d\db (1-\zeta_R)\Psi(r^{2\kappa})|\leq |\nabla^2 \zeta_R||r^{2\kappa}|+ |\nabla^2\Psi(r^{2\kappa})||1-\zeta_R|+|\nabla \zeta_R||\nabla r^{2\kappa}|\leq C R^{2(\kappa-1)}\ll 1\]
	Then $\hat{\omega}_0+i\d\db\psi_{\eps} > 0$ and $\hat{\omega}_0+i\d\db\psi_{\eps} > \eps_0\omega$ on the compact set $K$ containing $V$, hence there exist an $\eps>0$ such that $\hat{\omega}_0+i\d\db\psi_{\eps} > \eps\omega$ holds. 
\end{proof}

Now we prove the main estimate of this section. 
\begin{Proposition}\label{prop: C2-bdd}
    There are uniform constants $B, C>0$, independent of $t$ such that the following estimate holds:
    \begin{equation*}
    	|\d\db\varphi_t| \leq Ce^{-B\psi_{\eps}}.
    \end{equation*}
\end{Proposition}

\begin{proof}
By the well-known computation of Aubin and Yau, we have
\begin{equation*}
    \lapl_{\varphi_t}\log \Tr_{\omega}\omega_{\varphi_t} \geq -A\Tr_{\varphi_t}\omega
\end{equation*}
where $A$ is a lower bound for the bisectional curvatures of $\omega$. 
Then if we pick $N\gg B$ sufficiently large, we have
\begin{align*}
	\lapl_{\varphi_t}\left(\log \Tr_{\omega}\omega_{\varphi_t}+
	B\psi_{\eps}-N\varphi_t\right)&\geq (B\eps-A)\Tr_{\varphi_t}\omega-B\Tr_{\varphi_t}\hat{\omega}_0+N\Tr_{\varphi_t}\hat{\omega}_t-Nn\\
	&\geq C\left(\frac{\omega^n}{ni^{n^2}\Omega\wedge\bar{\Omega}}\Tr_{\omega}\omega_{\varphi_t}\right)^{\frac{1}{n-1}}- Bn
\end{align*}
since $\psi_{\eps}$ goes to $-\infty$ near $V$ and the function $\log \Tr_{\omega}\omega_{\varphi_t}+
B\psi_{\eps}-N\varphi_t$ goes to 0 at infinity, either $\log \Tr_{\omega}\omega_{\varphi_t}+
B\psi_{\eps}-N\varphi_t$ is always non-positive, in which case we are done, or maximum is achieved in the interior, and applying the maximum principle gives
\begin{equation*}
	\Tr_{\omega}\omega_{\varphi_t}\leq Ce^{B(\sup\psi_{\eps}-\psi_{\eps})}
\end{equation*}
from which the estimate follows.
\end{proof}

\begin{Remark}
	This argument is the only place where we used the K\"ahler current in Assumption~\ref{ass: mainAss}. In the situation where $[\alpha_0] = \pi^{*}c_1(L)$ where $\pi:X\to X_0$ is a crepant resolution of a singular Calabi-Yau variety with compactly supported singularities and $L\to X_0$ is an ample line bundle on $X_0$, the above $C^2$ estimate can be replaced by the argument in Lemma~\ref{Schwartz Lemma}, and the convergence holds away from $\pi^{-1}(X_0^{sing})$. In that case we do not need the K\"ahler current in Assumption~\ref{ass: mainAss} to prove Theorem~\ref{Main Theorem}. 
\end{Remark}

The higher order estimates follow from the standard methods of Yau \cite{Yau78, PSS, SW}. 

\begin{Proposition}[Higher order estimates]\label{prop: C-inf-loc-bdd}
	We have a uniform estimate
	\begin{equation*}
		\|\varphi_t\|_{C^{k, \alpha}_{loc}(K)}\leq C(K, k, \alpha)
	\end{equation*}
	for any $K\subset \subset X\setminus V$ and $C$ independent of $t$. 
\end{Proposition}
\begin{proof}
	This follows from the local estimates in \cite{SW}. 
\end{proof}
\begin{Corollary}\label{cor: C-infty-loc convergence of metrics}
	The metrics $\omega_{\varphi_t}$ converge after passing to a subsequence in $C^{\infty}_{\text{loc}}(X\setminus V)$ to a possibly incomplete metric $\omega_{\varphi_0}$ on $X\setminus V$, which is uniformly equivalent to $\omega_C$ at infinity. 
\end{Corollary}

So far, we've shown the first two parts of Proposition~\ref{estimates}, in the next section we prove decay estimates for $\varphi_t$. 

\subsection{Decay estimates}
In this section, we prove uniform decay estimates for $\varphi_t$. We use the method of Moser iteration with a weight, similar to the technique used in \cite[Chap 8]{Joyce}.  However, as in Section~\ref{subsec: C0bd}, we use the Ricci flat metrics $\omega_{\varphi_t}$, exploiting the uniform control of the Sobolev constants. 

Recall that $r:X\to \R_{> 0}$ is a radius function such that $|\nabla r|+r|i\d\db r|\leq C$, and it's not hard to see that we can also assume that $r = const$ on a compact set $K$ containing the singular set $V$. 

\begin{Definition}\label{weighted-lp-norm}
	We define the following weighted $L^p$ norms, 
	\begin{equation*}
		\|u\|_{L^{p}_{\delta}({i^{n^2}\Omega\wedge\bar{\Omega}})} = \left(\int_X |u r^{\delta}|^p r^{-2n}i^{n^2}\Omega\wedge\bar{\Omega}\right)^{\frac{1}{p}}
	\end{equation*}
\end{Definition}

\begin{Remark}
	Notice if we let $p\to \infty$, then the $L^p_{\delta}$ norms converge to the $L^{\infty}_{\delta}$ norm given by $\|u\|_{L^{\infty}_{\delta}} = \sup_{X}|ur^{\delta}|$
\end{Remark}

\begin{Proposition}\label{weighted-lp-bound}
    For any $\delta<\gamma$, we have a uniform bound of the form
    \begin{equation*}
        \|\varphi_t\|_{L^{p}_{\delta}(\omega_{\varphi_t}^n)}\leq C
    \end{equation*}
    for any $p\in (0, \frac{2n}{\delta}]$, and constant depending on $p, \delta$. 
\end{Proposition}
\begin{proof}
If $p = \frac{2n}{\delta}$, then this is simply the $L^{\frac{2n}{\delta}}$ norm, which is bounded if $\delta<\gamma$ by Proposition \ref{lp-bound}. 
If $p<\frac{2n}{\delta}$, then 
\begin{equation*}
    \int_X|\varphi r^{\delta}|^pr^{-2n}\omega_{\varphi_t}^n\leq \left(\int_X|\varphi_t|^{pq}\right)^{\frac{1}{q}}\left(\int_Xr^{\frac{q}{q-1}(\delta p-2n)}\omega_{\varphi_t}^n\right)^{\frac{q-1}{q}}
\end{equation*}
the first term is bounded if $q>\frac{2n}{\gamma p}$ by Proposition \ref{lp-bound}, and the second term is finite if $q<\frac{2n}{\delta p}$, so we just need to pick 
$q\in (\frac{2n}{\gamma p}, \frac{2n}{\delta p})$ with $q>1$, which is possible since $p<\frac{2n}{\delta}$. 
\end{proof}

\begin{Proposition}
    For any $\delta<\gamma$, $p>1$, we have
    \begin{equation*}
        \|\varphi_t r^{\delta}\|^p_{L^{p\frac{n}{n-1}}(r^{-2n}\omega_{\varphi_t}^n)}\leq \frac{Cp^2}{p-1}\left(\|\varphi_tr^{\delta}\|^{p-1}_{L^{p-1}(r^{-2n}\omega_{\varphi_t}^n)}+\|\varphi_tr^{\delta}\|^p_{L^{p}(r^{-2n}\omega_{\varphi_t}^n)}\right)
    \end{equation*}
    for C depending on the Sobolev constant of $\omega_{\varphi_t}$, $\delta$ and the dimension $n$. 
\end{Proposition}
\begin{proof}
We use the same method as in \cite[Proposition 8.6.7]{Joyce}, but using the Calabi-Yau metrics $\omega_{\varphi_t}$ as the background metrics. The reason is because the metrics $\omega_{\varphi_t}$ are Ricci-flat and hence have a uniform Sobolev inequality.
First we set
\[T_t = \sum_{k=0}^{n-1}\omega_{\varphi_t}^k\wedge\hat{\omega}_t^{n-1-k}. \]
If $q-p\gamma<-2n+2 $, then Stoke's theorem gives the following two identities
{\small
\begin{align*}
0 &= \int_Xi\d\left(r^q|\varphi_t|^{p-2}\varphi_t\db\varphi_t\wedge T_t\right)\\  &= (p-1)\int_Xr^q|\varphi_t|^{p-2}i\d\varphi_t\wedge\db\varphi_t\wedge T_t+q\int_Xr^{q-1}|\varphi_t|^{p-2}\varphi_ti\d r\wedge\db\varphi_t\wedge T_t+ \int_Xr^q|\varphi_t|^{p-2}\varphi_ti\d\db\varphi_t\wedge T_t
\end{align*}
}
and
\begin{align*}
0 &= -\int_Xi\db\left(r^{q-1}|\varphi_t|^pi\d r\wedge T_t\right)\\
&= p\int_Xr^{q-1}|\varphi_t|^{p-2}\varphi_ti\d r\wedge\db\varphi_t\wedge T_t + (q-1)\int_Xr^{q-2}|\varphi_t|^pi\d r\wedge\db r\wedge T_t + \int_Xr^{q-1}|\varphi_t|^pi\d\db r\wedge T_t
\end{align*}
using these identities, we can obtain through integration by parts
\begin{align*}
\int_{X}|\nabla(|\varphi_t|^{\frac{p}{2}}r^{\frac{q}{2}})|^2_{\omega_{\varphi_t}}\omega_{\varphi_t}^n &= n\int_{X}i\d(|\varphi_t|^{\frac{p}{2}}r^{\frac{q}{2}})\wedge\db(|\varphi_t|^{\frac{p}{2}}r^{\frac{q}{2}}) \wedge \omega_{\varphi_t}^{n-1}\\
&\leq  n\int_{X}i\d(|\varphi_t|^{\frac{p}{2}}r^{\frac{q}{2}})\wedge\db(|\varphi_t|^{\frac{p}{2}}r^{\frac{q}{2}}) \wedge T_t\\
&= -\frac{np^2}{4(p-1)}\int_X\varphi_t|\varphi_t|^{p-2}r^{q}i\d\db\varphi_t\wedge T_t\\
& \qquad + \frac{mq}{4(p-1)}\int_X|\varphi_t|^pr^{q-2}[(p+q-2)i\d r\wedge\db r-(p-2)r i\d\db r]\wedge T_t\\
&= -\frac{np^2}{4(p-1)}\int_X\varphi_t|\varphi_t|^{p-2}r^{q}(e^{f_t}-1)\omega_{\varphi_t}^n\\
& \qquad + \frac{mq}{4(p-1)}\int_X|\varphi_t|^pr^{q-2}[(p+q-2)i\d r\wedge\db r-(p-2)r i\d\db r]\wedge T_t
\end{align*}
where in the last equality, we used the equation $i\d\db\varphi_t\wedge T_t = (e^{f_t}-1)\omega_{\varphi_t}^n$. Now we claim there also exist a uniform constant $C$ independent of $t$ and $r$ such that
\[\left|\frac{[(p+q-2)i\d r\wedge\db r-(p-2)r i\d\db r]\wedge T_t}{i^{n^2}\Omega\wedge\bar{\Omega}}\right|\leq C(p+|q|)\] 
recall that we chose $r$ so that $r = const$ on a compact set $K$ containing $V$, so the left hand side of the expression is $0$ on $K$. By Corollary~\ref{cor: C-infty-loc convergence of metrics} we know that $|T_t|\leq C$ on $X\setminus K$ and because $r$ is a radius function, we also $|\nabla r|+r|\d\db r|\leq C$, putting them together, we get that the expression also holds on $X\setminus K$, hence this whole expression is bounded by the right hand side. 

This then combined with the Sobolev inequality, we conclude that 
\begin{equation*}
    \left(\int_{X}|\varphi_t|^{p\frac{n}{n-1}}r^{q\frac{n}{n-1}}\omega_{\varphi_t}^n\right)^{\frac{n-1}{n}}\leq \frac{Cnp^2}{4(p-1)}\int_X|e^{-f_t}-1||\varphi_t|^{p-1}r^q\omega_{\varphi_t}^n + \frac{Cq(p+q)}{4(p-1)}\int_X|\varphi_t|^pr^{q-2}\omega_{\varphi_t}^n
\end{equation*}
for any $\delta<\gamma$ we can set $q = 2(1-n)+p\delta$ and use the fact that $|e^{f_t}-1|\leq Cr^{-\gamma-2}$ to obtain, 
\begin{align*}
    \left(\int_X|\varphi_tr^{\delta}|^{p\frac{n}{n-1}}r^{-2n}\omega_{\varphi_t}^n\right)^{\frac{n-1}{n}}&\leq C\frac{p^2}{4(p-1)}\left(\int_X |\varphi_t|^{p-1}r^{p\delta-\gamma}r^{-2n}\omega_{\varphi_t}^n+\int_X|\varphi_tr^{\delta}|^pr^{-2n}\omega_{\varphi_t}^n \right)\\
    & = C\frac{p^2}{4(p-1)}\left(\int_X |\varphi_tr^{\delta}|^{p-1}r^{\delta-\gamma}r^{-2n}\omega_{\varphi_t}^n+\int_X|\varphi_tr^{\delta}|^pr^{-2n}\omega_{\varphi_t}^n \right)
\end{align*}
and since $\delta<\gamma$, which means for any $p>1$, we have
\begin{equation*}
    \|\varphi_tr^{\delta}\|^p_{L^{p\frac{n}{n-1}}(r^{-2n}\omega_{\varphi_t}^n)}\leq \frac{Cp^2}{p-1}\left(\|\varphi_tr^{\delta}\|^{p-1}_{L^{p-1}(r^{-2n}\omega_{\varphi_t}^n)}+\|\varphi_tr^{\delta}\|^p_{L^{p}(r^{-2n}\omega_{\varphi_t}^n)}\right)
\end{equation*}
\end{proof}

\begin{Corollary}\label{C0-decay}
    For any $\delta<\gamma$, we have a uniform bound of the form
    \begin{equation*}
        |\varphi_t|\leq Cr^{-\delta}
    \end{equation*}
    for $C$ depending on $\delta$. 
\end{Corollary}
\begin{proof}
    By Proposition \ref{weighted-lp-bound}, we have a weighed $L^p$ bound for any $p\leq \frac{2n}{\delta}$, combined with the previous proposition, we can use the standard Moser iteration argument starting from $p=\frac{2n}{\delta}\geq \frac{n}{n-1}>1$. 
\end{proof}

\begin{Proposition}\label{high-decay}
	For any $\delta<\gamma$, the derivative of the solutions $\varphi_t$ satisfy uniform decay estimates on $X\setminus K$,
	\begin{equation*}
		|\nabla^k \varphi_t|\leq Cr^{-\delta-k}
	\end{equation*}
	where $C = C(n, \delta, k)$ which doesn't depend on $t$. 
\end{Proposition}
\begin{proof}
	This follows from the methods of \cite[Theorem 8.6.11]{Joyce} verbatim. The point to note here is that the metrics $\omega_{\varphi_t}$ are uniformly equivalent to $\omega_C$ on the region $X\setminus K$, with bounded derivatives as well, hence the Schauder constants are uniformly controlled on far away balls. 
\end{proof}

\begin{Proposition}\label{optimal-decay}
	If $\gamma\in (0, 2n-2)$, then in fact we have
	\begin{equation*}
		|\nabla^k \varphi_t|\leq Cr^{-\gamma-k}
	\end{equation*}
	on $X\setminus K$, and $C=C(n, k)$ independent of $t$. 
\end{Proposition}
\begin{proof}
	This follows from the same argument as in \cite[Chap 8.7, Theorem A2]{Joyce}. 
\end{proof}

We can now prove Proposition~\ref{estimates}, thereby completing the proof of Theorem~\ref{Main Theorem}.

\begin{proof}[Proof of Proposition~\ref{estimates}]
Combine Corollary~\ref{cor: C0-bdd}, Proposition~\ref{prop: C-inf-loc-bdd}, Proposition~\ref{C0-decay} and Proposition~\ref{optimal-decay}. 

\end{proof}

We now prove the local diameter bound, which will play an important role throughout the remainder of the paper.

\begin{Lemma}\label{diam-bdd}
	In the setting of Theorem~\ref{Main Theorem}, let $K\subset X$ be a compact subset containing $V$.   Then the diameter of $K$ with respect to the Calabi-Yau metrics $\omega_{t,CY}$ is uniformly bounded from above as $t\rightarrow 0$. 
	\begin{equation*}
	\text{Diam}_{\omega_{\varphi_t}}K\leq C
	\end{equation*}
\end{Lemma}
\begin{proof}
	It suffices to show that the sets $K_R = \{r(x)\leq R\}$ have bounded diameters for $R$ sufficiently large. Recall that the metrics $\omega_{\varphi_t}$ are uniformly asymptotic to $\omega_{cone}$ for $r$ large and $t$ close to $0$ by Proposition~\ref{high-decay}. Fix any two points $x, y\in K_R$,  and joint them by a length minimizing geodesic $\gamma:[0, L]\to X$. We claim that $\gamma$ must lie inside $K_{R^2}$ for $R$ sufficiently large. Note for $R$ large, on the region $\{r(x)\geq R\}$ the metric $\omega_{\varphi_t}$ is $C^{\infty}$ close to a cone metric uniformly in $t$, and hence for $R$ sufficiently large, the boundary of $K_R$ has diameter bounded by $2 \pi R$. However, the distance between the boundary of $K_{R}$ and $K_{R^2}$ on the order of $R^2$, so it's clear that any minimizing geodesic between two points in $K_R$ cannot leave $K_{R^2}$. Now consider $x_i = \gamma(2i+1)$ and disjoint balls $B_1(x_i)$.  Note that these balls have a fixed lower bound on the volume, since by Bishop-Gromov volume comparison and the asymptotically conical geometry we have
	\[
	{\rm Vol}(B_1(x_i)) \geq \lim_{S\rightarrow \infty} \frac{{\rm Vol}(B_{S}(x_i))}{S^{2n}} = {\rm Vol}_{g_C}(L) =: c >0
	\]
	where $L$ is the link of the cone at infinity, identified with $\{r_{C}=1\}$ and $g_{C}$ is the conical Calabi-Yau metric. Thus, we have
	\[
	\sum_{i}\Vol(B_1(x_i))\geq c\frac{\lfloor L\rfloor}{2}
	\]
	where $c$ is the non-collapsing constant. On the other hand, these balls must all lie in $K_{2R}$, and since the volume form of the Calabi-Yau metrics are fixed, we must have that 
	\[ 
	c\frac{\lfloor L\rfloor}{2} \leq \int_{K_{R^2}}i^{n^2}\Omega\wedge\bar{\Omega}
	\]
	which gives us a bound for $L$, which is $d_{\omega_{\varphi_t}}(x, y)$. 
\end{proof}

\subsection{Uniqueness}
In this section, we discuss the uniqueness of the Calabi-Yau currents constructed in the previous sections. 


\begin{Theorem}\label{thm: uniqueness}
	The current that we constructed $\omega_{\varphi_0}$ above is unique in the sense that if $\omega$ is another positive current with locally bounded potentials in the same cohomology class as $\omega_{\varphi_0}$ which is smooth on $X\setminus V$, asymptotically conical at infinity with any rate $\delta>0$ and satisfies the complex Monge-Amp\`ere equation
	\begin{equation*}
		\omega^n = \omega_{\varphi_0}^n = i^{n^2} \Omega\wedge\bar{\Omega}
	\end{equation*}
	in the Bedford-Taylor sense, then $\omega = \omega_{\varphi_0}$. 
\end{Theorem}

The proof is modelled after the idea introduced in \cite{CH13}, which relies on the following crucial Lemma proved in \cite{CH13}. 

\begin{Lemma}\label{lem: subquadratic harmonic functions are pluriharmonic}\cite[Corollary 3.9]{CH13}
	Suppose $(X, \omega)$ is an asymptotically conical K\"ahler manifold with $\Ric \geq 0$, then for any $\eps>0$, any harmonic function $u\in C^{\infty}_{2-\eps}(X)$ is pluriharmonic. 
\end{Lemma}

The idea is to write $\omega = \omega_{\varphi_0}+i\d\db \psi$ and use this lemma to improve the asymptotics of the potential function $\psi$ by subtracting off pluriharmonic functions from it, until we are left in the case where the potential function is decaying in which case uniqueness follows from a standard integration by parts argument. 

\begin{Proposition}\label{prop: weak uniqueness}
Suppose $\varphi\in PSH(X, \omega_{\varphi_0})\cap L^{\infty}(X)\cap C^{\infty}_{-\eps}(X\setminus V)$ is a function such that the current $\omega_{\varphi_0}+i\d\db\varphi$ satisfies
\begin{equation*}
	(\omega_{\varphi_0}+i\d\db\varphi)^n = \omega_{\varphi_0}^n = i^{n^2}\Omega\wedge\bar{\Omega}
\end{equation*}
in the Bedford Taylor sense, then $\varphi = 0$. 
\end{Proposition}

\begin{proof}
{\small
\begin{align*}
0 &= -\int_{B_R} |\phi|^{p-2}\phi((\omega_{\varphi_0}+i\d\db\varphi)^n-\omega_{\varphi_0}^n) = -\int_{B_R}|\phi|^{p-2}\phi i\d\db\phi\wedge\left(\sum_{k=0}^{n-1}\omega_{\varphi_0}^k\wedge(\omega_{\varphi_0}+i\d\db\varphi)^{n-1-k}\right)\\
&= \frac{4(p-1)}{p^2}\int_{B_R}i\d(|\phi|^{\frac{p}{2}})\wedge\db(|\phi|^{\frac{p}{2}})\wedge\left(\sum_{k=0}^{n-1}\omega_{\varphi_0}^k\wedge(\omega_{\varphi_0}+i\d\db\varphi)^{n-1-k}\right)\\
&\qquad -\int_{\d B_R}|\phi|^{p-2}\phi i\db\phi\wedge\left(\sum_{k=0}^{n-1}\omega_{\varphi_0}^k\wedge(\omega_{\varphi_0}+i\d\db\varphi)^{n-1-k}\right)
\end{align*}
}
picking $p>\frac{2n-2}{\gamma}$ and letting $R\to \infty$, we get 
\[\int_{X}i\d(|\phi|^{\frac{p}{2}})\wedge\db(|\phi|^{\frac{p}{2}})\wedge\left(\sum_{k=0}^{n-1}\omega_{\varphi_0}^k\wedge(\omega_{\varphi_0}+i\d\db\varphi)^{n-1-k}\right) = 0\]
which shows that $\phi = 0$. 
\end{proof}

\begin{Lemma}\label{lem: improving the potential}
	Suppose $(X, J, g)$ is an asymptotically conical Calabi-Yau manifold with rate $\nu>0$, and $\eta = \eta_{i\bj}$ is a asymptotically conical hermitian metric with rate $\nu>0$ and let $u\in C^{\infty}_{2-\beta}$ such that $\eta^{i\bj}u_{i\bj} \in C^{\infty}_{-\kappa}$, then there exist $\tilde{u}\in C^{\infty}_{2-\beta-\nu}$ such that $i\d\db \tilde{u} = i\d\db u$. 
\end{Lemma}
\begin{proof}
	We have 
	\begin{equation*}
		g^{i\bj}u_{i\bj} = (g^{i\bj}-\eta^{i\bj})u_{i\bj}+\eta^{i\bj}u_{i\bj}\in C^{\infty}_{-\min(\kappa, \beta+\nu)}
	\end{equation*}
	hence we can solve the equation $g^{i\bj}\tilde{u}_{i\bj} = g^{i\bj}u_{i\bj}$ with $\tilde{u}\in C^{\infty}_{2-\min(\kappa, \beta+\nu)}$ and by Lemma~\ref{lem: subquadratic harmonic functions are pluriharmonic} we have $i\d\db\tilde{u} = i\d\db u$.  
\end{proof}

\begin{proof}[Proof of Theorem~\ref{thm: uniqueness}]
	By the $\d\db$-Lemma (Proposition~\ref{ddb-lemma}), we can write $\omega = \omega_{\varphi_0}+i\d\db\psi$, for $\psi\in PSH(X, \omega_{\varphi_0})\cap L^{\infty}_{loc}(X)\cap C^{\infty}_{loc}(X\setminus V)$, then choose a cutoff $\chi$ such that $\chi$ has compact support and $\chi = 1$ on a compact set $K$ containing $V$, then since $i\d\db \psi = \omega-\omega_{\varphi_0} \in C^{\infty}_{-\eps}(X\setminus V)$ for some $\eps>0$, hence by Proposition~\ref{prop: quan-ddbar-lemma}, we can solve $i\d\db f = i\d\db[(1-\chi)\psi]$ for $f\in C^{\infty}_{\gamma}$, $\gamma =2-\eps$. Setting $\varphi = \chi \psi+f$, we have that $\varphi\in L^{\infty}_{loc}(X)\cap C^{\infty}_{\gamma}(X\setminus V)$ and 
	\[ (\omega_{\varphi_0}+ i\d\db \varphi)^n = \omega_{\varphi_0}^n= i^{n^2}\Omega\wedge\bar{\Omega}\]
	If $\gamma<0$, then we are done by Proposition~\ref{prop: weak uniqueness}. If $\gamma>0$, then we proceed by the following: note that the equation above can be rewritten as	
	\[\lapl_{\omega_{\varphi_0}} \varphi  = -(i\d\db \varphi)^2\wedge\left(\sum_{k=2}^n {n\choose k} \frac{(i\d\db\varphi)^{k-2}\wedge\omega_{\varphi_0}^{n-k}}{\omega_{\varphi_0}^n}\right)\in C^{\infty}_{2\gamma-4}(X\setminus V)\]
	if $\chi$ is the cutoff function as before, then we have $\lapl_{\omega_{\varphi_0}}[(1-\chi)\varphi]\in C^{\infty}_{2\gamma-4}(X)$ if we let $\eta = \chi \omega_{\varphi_t} + (1-\chi)\omega_{\varphi_0}$, then $\eta$ is an asymptotically conical hermitian metric which is equal to $\omega_{\varphi_0}$ outside of a compact set, hence $\eta^{i\bj}[(1-\chi)\varphi]_{i\bj}\in C^{\infty}_{2\gamma-4}(X)$, hence we can apply Lemma~\ref{lem: improving the potential} with $\kappa = 2(2-\gamma)$ and $\beta = 2-\gamma$, so we can solve $i\d\db v = i\d\db [(1-\chi)\varphi]$ with $v\in C^{\infty}_{\gamma-\min(2-\gamma,\nu)}$ now we can set $\tilde{\varphi} = v+\chi \varphi\in C^{\infty}_{\gamma-\min(2-\gamma,\nu)}(X\setminus C)$ and we can keep repeating this process with $\tilde{\varphi}$ in place of $\varphi$ and $\gamma-\min(2-\gamma,\nu)$ in place of $\gamma$ until are in the case where $\gamma<0$, then we are done by Proposition~\ref{prop: weak uniqueness}.

\end{proof}

\section{Metric geometry of the singular Calabi-Yau}\label{sec: metricGeometry}
The goal of this section is to prove Theorem \ref{Theorem2}. Let us first begin with some definitions and the general setup. 
\begin{Definition}
	We say that a complex analytic space $X_0$ is a {\it singular Calabi-Yau variety with compactly supported, crepant singularities}, if
	\begin{itemize}
	\item $X_0$ is normal singularities, Gorenstein and log-terminal,
	\item there is a compact set $K$ so that $X_0\backslash K$ is smooth,
	\item there exists a resolution $\pi:X\to X_0$ such that $X$ also has trivial canonical bundle and $\pi^{*}\Omega$ extends as a non-vanishing global holomorphic $(n, 0)$-form on $X$. (By abuse of notation, we will also denote this holomorphic $(n, 0)$-form by $\Omega$) 
	\end{itemize}
\end{Definition}

Let $X_0$ be a singular Calabi-Yau variety with compactly supported, crepant singularities. Suppose that the resolution $(X, J, \Omega)$ is K\"ahler and it has a K\"ahler metric $\omega$ such that $(X, J, \omega, \Omega)$ is asymptotic to a Calabi-Yau cone $(C, J_C, \omega_C, \Omega_C)$ at rate $\nu$. 

\begin{Definition}
	A line bundle $L$ on $X_0$ is ample if for some $k>0$, there exist sections $s_0, \ldots, s_N \in H^0(X_0, L^k)$ such that $[s_0, \ldots, s_N]$ gives an embedding of $X_0$ into a finite dimensional projective space $\C P^N$, and denote this embedding map by $\iota$, then we have $\frac{1}{k}[\iota^{\star}\omega_{FS}] = c_1(L)$. 
\end{Definition}

Let us now fix $L$ an ample line bundle on $X_0$. If set $[\alpha_0] = \pi^{*}c_1(L)$, then suppose $(X, J, \omega, \Omega)$ and $[\alpha_0]$ satisfy the hypothesis of Theorem \ref{Theorem2}. Then from the previous sections, we have on $X$, a sequence of Calabi-Yau metrics $\omega_{\varphi_t} = \hat{\omega}_t+i\d\db\varphi_t$ with $[\omega_{\varphi_t}] =(1-t)[\alpha_0]+t[\alpha_1]$, which satisfy the equation
\begin{equation*}
	(\hat{\omega}_t+i\d\db\varphi_t)^n = e^{f_t}\hat{\omega}_{t}^n ( = i^{n^2}\Omega\wedge\bar{\Omega})
\end{equation*}
and $f_t = \log \frac{i^{n^2}\Omega\wedge\bar{\Omega}}{\hat{\omega}_0^n} \in C^{\infty}_{-\gamma-2}(X)$, and $\varphi_t\in C^{\infty}_{-\gamma}(X)$. 

If we fix a point $p\in \pi^{-1}(X_0^{reg})$, then by Gromov compactness, after passing to a subsequence, the pointed spaces $(X, \omega_{\varphi_{t_i}}, p)$ for $t_i\to 0$ pointed Gromov-Haussdorff converge to a limiting pointed metric space $(X_{\infty}, d_{\infty}, p_{\infty})$ as $i\to \infty$. By the definition of pointed Gromov-Haussdorff convergence, the convergence can be interpreted in the following sense: 
If we set $Z = (X_{\infty}, d_{\infty}, p_{\infty})\sqcup \bigsqcup_{t_i}(X, \omega_{\varphi_{t_i}}, p)$, then there exist a metric $d_Z$ on $Z$ such that 
\begin{enumerate}
	\item $d_Z|_{X_i} = d_{g_{\varphi_{t_i}}}$ 
	\item $d_Z(\underbrace{p}_{\in X_i}, p_{\infty})\to 0$ 
	\item $B_{g_{\varphi_{t_i}}}(p, r)\subset X_i\to B_{g_{\infty}}(p_{\infty}, r)\subset X_{\infty}$ in the Haussdorff sense with respect to $d_Z$. 
\end{enumerate}
 The asymptotically conical property of $\omega_{\varphi_t}$ implies that the tangent cone at $\infty$ is independent of $t$, and by Bishop-Gromov, this a uniform lower bound on volume, and we have $\Vol_{\omega_{\varphi_t}}B(p, r)\geq cr^{2n}$ where $c$ is the volume ratio of the asymptotic cone $C$. Hence the regularity theory of Cheeger, Colding and also Tian \cite{CC0, CC1, CC2, CC3, CCT} applies, and the limiting space admits the following structure
\begin{enumerate}
	\item All tangent cones of $X$ are metric cones. 
	\item $X = \mathcal{R}\cup \mathcal{S}$, where $\mathcal{R}$ consists of all the points where all tangent cones are isometric to $\R^{2n}$. 
	\item $\mathcal{R}$ is an open dense set in $X_{\infty}$ with a smooth metric $g_{\infty}$ and complex structure $J_{\infty}$ which makes it Ricci-flat K\"ahler manifold and $(X_{\infty}, d_{\infty}) = \overline{(\mathcal{R}, d_{g_{\infty}})}$. Moreover, the convergence of $(X, J, \omega_{\varphi_t}, p)\to (X_{\infty}, J_{\infty}, g_{\infty}, p)$ is smooth on $\mathcal{R}$ in the sense that for every $K\subset \subset \mathcal{R}$, there exist smooth maps $\eta_{i}: K\to X$ such that $(\eta_i^{\star}g_{t_i}, \eta_i^{\star}J)$ converges to $(g_{\infty}, J_{\infty})$ smoothly on $K$. (In fact, we can arrange $\eta_i$ such that $d_Z(\eta_i(z), z)\to 0$ uniformly in $K$)
	\item $\mathcal{S}$ is a closed subset of $X_{\infty}$ with real Hausdorff codimension greater or equal to $4$. 
\end{enumerate}

\subsection{Properties of the Gromov-Hausdorff limit}
In this section, we prove several preliminary propositions about the relationship between $X_{\infty}$ and the K\"ahler current constructed from Theorem \ref{Main Theorem}. In particular, we show the following:
\begin{enumerate}
	\item $\omega_{\varphi_0}$ is in fact well-defined and smooth on $\pi^{-1}(X_0^{reg})$
	\item There exist a locally isometric embedding of $\iota_{\infty}: (\pi^{-1}(X_0^{reg}), \omega_{\varphi_0})\to (\mathcal{R}, g_{\infty})$. 
	\item $X_{\infty}$ is isometric to the metric completion $\overline{(\pi^{-1}(X_0^{reg}), \omega_{\varphi_0})}$
	\item $\iota_{\infty}$ is a bijective local isometry between $X_0^{reg}$ and $\mathcal{R}$. 
\end{enumerate}

One of the key ingredients is the local diameter bound Lemma~\ref{diam-bdd}, which we apply with $V= \pi^{-1}(X_0^{sing})$.

\begin{Proposition}\label{Schwartz Lemma}
	The family of metrics $\omega_{\varphi_t}$ has a uniform lower bound
	\begin{equation}\label{lower-bdd}
	\omega_{\varphi_t}\geq \frac{1}{C}\hat{\omega}_{0}
	\end{equation}
\end{Proposition}
\begin{proof}
	By the standard Schwartz lemma calculation, we have
	\[\lapl_{\omega_{\varphi_t}}\log \Tr_{\omega_{\varphi_t}}\pi^{\star}\omega_{FS}\geq -4\Tr_{\omega_{\varphi_t}}\pi^{\star}\omega_{FS}\]
	and for any other K\"ahler metric $\hat{\omega}$, one also has
	\[\lapl_{\omega_{\varphi_t}}\log \Tr_{\omega_{\varphi_t}}\hat{\omega}\geq -C\Tr_{\omega_{\varphi_t}}\hat{\omega}\]
	with $C$ depending only on the upper bound for the holomorphic bisectional curvature of $\hat{\omega}$. 
	Recall from the construction of $\hat{\omega}_0$ in Propositions~\ref{bkground-metric 1} and~\ref{prop: bkground-metric-improvement} that $\hat{\omega}_0$ can be taken  to be  equal to $\frac{1}{k}\pi^{\star}\omega_{FS}$ on a compact set $K$ containing $\pi^{-1}(X_0^{sing})$, and is a genuine non-degenerate, asymptotically conical K\"ahler metric outside of $K$, so we can apply the first inequality inside $K$ and the second outside $K$ to get a uniform estimate
	\begin{equation}\label{Schwartz inequality}
	\lapl_{\omega_{\varphi_t}}\log \Tr_{\omega_{\varphi_t}}\hat{\omega}_0 \geq -C\Tr_{\omega_{\varphi_t}}\hat{\omega}_0. 
	\end{equation}
	Since $\omega_{\omega_{\varphi_t}}=\hat{\omega}_{t}+i\d\db\varphi_t$, taking trace gives 
	\[
	n = \Tr_{\omega_{\varphi_t}}\hat{\omega}_t+\lapl_{\varphi_t}\varphi_t,
	\]
	and we also know that for $t$ reasonably small $\hat{\omega}_{t}\geq c\hat{\omega}_0$ holds for some small constant $c$ uniformly in $t$ as $t\to 0$, which means we have
	\[
	n\geq c\Tr_{\varphi_t}\hat{\omega}_0+\lapl_{\varphi_t}\varphi_t.
	\]
	Combining this with \eqref{Schwartz inequality}, we have
	\begin{equation*}
	\lapl_{\omega_{\varphi_t}}\left(\log \Tr_{\omega_{\varphi_t}}\hat{\omega}_{0} - A\varphi_t\right)\geq (\frac{Ac}{2}-C)\Tr_{\omega_{\varphi_t}}\hat{\omega}_{0}-An 
	\end{equation*}
	since $\log \Tr_{\omega_{\varphi_t}}\hat{\omega}_{0} - A\varphi_t$ converges to the constant $\log n$ at spacial infinity, if the maximum is attained at infinity, then we automatically have a uniform bound that we wanted. So we can assume the maximum is achieved in the interior, and applying the maximum principle to the equation above, and we obtain 
	\[\Tr_{\omega_{\varphi_t}}\hat{\omega}_{0}\leq Ce^{A(\varphi_t-(\varphi_t)_{min})}\] which gives a uniform upper bound for $\Tr_{\omega_{\varphi_t}}\hat{\omega}_{0}$. 
\end{proof}
\begin{Corollary}
	On $X\setminus \pi^{-1}(X_0^{sing})$, we have
	\begin{equation*}
		C^{-1}\hat{\omega}_0 \leq \omega_{\varphi_t}\leq Ce^{f_0}\hat{\omega}_0
	\end{equation*}
	where $e^{f_0} = \frac{i^{n^2}\Omega\wedge\bar{\Omega}}{\hat{\omega}_0^n}$ is bounded uniformly away from $\pi^{-1}(X_0^{sing})$. In particular, this implies that $\omega_{\varphi_0}$ is smooth on $\pi^{-1}(X_0^{reg})$, and on $X_0$ it is a K\"ahler current since it dominates $\hat{\omega}_0$. 
\end{Corollary}
\begin{proof}
	The lower bound on $\omega_{\varphi_t}$ is the content of the previous lemma, and from that and the fact that $\omega_{\varphi_t}^n = i^{n^2}\Omega\wedge\bar{\Omega} = e^{f_0}\hat{\omega}_0^n$, the corollary follows immediately. 
\end{proof}

\begin{Corollary}
	The maps $\pi_i: (X, \omega_{\varphi_{t_i}}, p)\to (X_0, \hat{\omega}_0, p)$ are has bounded derivative, hence it is uniformly lipschitz and we can pass to a continuous surjective map from the Gromov-Haussdorff limit $\pi_{\infty}:(X_{\infty}, d_{X_{\infty}}, p_{\infty})\to X_0$. Furthermore, for any $q\in X_0^{reg}$, the preimage $\pi_{\infty}^{-1}(q)$ consists of a single point. 
\end{Corollary}

\begin{proof}
	The fact that the maps have bounded derivative follows from the estimate \eqref{lower-bdd}, and from this it follows from an Arzela-Ascoli type argument that after passing to a subsequence, the projection maps $\pi_i$ limit to a continuous surjective map $\pi_{\infty}: X_{\infty}\to X_0$. The map $\pi_{\infty}$ can be characterized in the following way: if we fix $h_i:(X, \omega_{\varphi_{t_i}})\to X_{\infty}$ an $\eps_i$-isometry for $\eps_i\to 0$, then for any sequence of points $q_i\in X$ with $\pi(q_i)\to q\in X_0$, and $h_i(q_i) \to q_{\infty}\in X_{\infty}$, we have $\pi_{\infty}(q) = q_{\infty}$. 
	
	To see that the preimage of $\pi^{-1}(q)$ for $q\in X_0^{reg}$ consists of a single point, suppose for contradiction that it consisted of two points $q_1, q_2\in X_{\infty}$ with $d_{X_{\infty}}(q_1, q_2) = d>0$ and $\pi_{\infty}(q_1) = \pi_{\infty}(q_2) = q\in X_0^{reg}$, then from the construction of $\pi_{\infty}$, there exist a sequences of points $q_1^i, q_2^i\in X$ such that $\pi_i(q_1^i)\to q$ and $\pi_i(q_2^i)\to q$ and $h_i(q_1^i) = q_1$ and $h_i(q_2^i) = q_2$.  Then from the fact that $\pi_i(q_1^i)\to q$ and $\pi_i(q_2^i)\to q$ and $q\in X_0^{reg}$, we know that $q_1^i \to \pi^{-1}(q)$ and $q_2^i \to \pi^{-1}(q)$ in $X$ since $\pi$ is a resolution of singularities of $X_0$, and $g_{t_i}\to g_{\infty}$ smoothly in a neighborhood of $q$, it follows that $d_{g_{t_i}}(q_1^i, q_2^i) \to 0$ as $i\to \infty$. But we also have
	\begin{align*}
	d_{X_{\infty}}(h_i(q_1^i), h_i(q_2^i))-\eps_i \leq 	d_{g_{t_i}}(q_1^i, q_2^i)
	\end{align*}
	since $h_i$ is an $\eps_i$-isometry. This is a contradiction, because $d_{X_{\infty}}(h_i(q_1^i), h_i(q_2^i))-\eps_i  \to d>0$ by our assumption. 
\end{proof}

\begin{Proposition}
	There is an embedding $i_{\infty}: (X_0^{reg}, \omega_{\varphi_0}, p)\hookrightarrow (\mathcal{R}, g_{\infty}, p)$, which is a locally isometric embedding, and $\pi_{\infty}\circ\iota_{\infty} = id$.  
\end{Proposition}

\begin{proof}
	We can simply take $\iota_{\infty} = \pi_{\infty}^{-1}|_{X_0^{reg}}$, which is well-defined by the previous proposition. It's clear that the image of $\iota_{\infty}$ is contained in the regular set $\mathcal{R}\subset X_{\infty}$ and that it is continuous, so it suffices to show that this map is a local isometry. To see this, we note that if $q\in X_0^{reg}$, then there exist an $\eps>0$ such that $B_{g_{t_i}}(q, \eps)\subset X_0^{reg}$ for all $i\gg 1$. It follows from the diameter estimate (c.f. Lemma \ref{diam-bdd}) that the points $h_i(\pi^{-1}(q))$ are uniformly bounded in $X_{\infty}$, hence after passing to a subsequence, it converge to some point $q_{\infty}\in X_{\infty}$, it's clear that $q_{\infty} = \iota_{\infty}(q)$ since $\pi_i(q) = q$. Since the points $\pi^{-1}(q)\in X_i$ have a uniform harmonic radius lower bound, hence $(B_{g_{t_i}}(\pi^{-1}(q), \eps), g_{\varphi_{t_i}})\xrightarrow{C^{\infty}} (B_{g_{\infty}}(q_{\infty}, \eps), g_{\infty})$ and  by the smooth convergence of $g_{\varphi_t}\to g_{\varphi_0}$, we also have $(B_{g_{t_i}}(\pi^{-1}(q), \eps), g_{\varphi_{t_i}})\xrightarrow{C^{\infty}} (B_{g_{\varphi_0}}(q, \eps), \omega_{\varphi_0})$, it is then clear from the construction of $\pi_{\infty}$ that it maps $(B_{g_{\infty}}(q_{\infty}, \eps), g_{\infty})$ isometrically onto $(B_{g_{\varphi_0}}(q, \eps), \omega_{\varphi_0})$. 
\end{proof}

The following Proposition follows from the same arguments as in \cite{RZ13}. We include a proof here for the convenience of the reader. 

\begin{Proposition}\label{prop: GH-limit-metric-completion-homeo}
	The subset $E = \mathcal{R} \setminus \iota_{\infty}(X_0^{reg})\subset \mathcal{R}$ is an analytic subset, hence of real codimension bigger than or equal to 2, and moreover $\overline{(X_0^{reg}, g_{\infty})} = X_{\infty}$. 
\end{Proposition}

\begin{proof}
	It suffices to show that the holomorphic maps $\pi: (X, \omega_{\varphi_t}, p)\to X_0\subset  (\C P^N, \omega_{FS}, p)$ limits to a holomorphic map $\pi_{\infty}|_{\mathcal{R}}:(\mathcal{R}, J_{\infty}, g_{\infty})\to X_0\subset \C P^N$. Assuming for now that this is the case, then $\mathcal{R}\setminus \iota_{\infty}(X_0^{reg}) = \pi_{\infty}|_{\mathcal {R}}^{-1}(X_0^{sing})$. 
	Since $X_0^{sing}\subset X_0$ is an analytic set, if $\pi_{\infty}|_{\mathcal{R}}$ is holomorphic, then $\pi_{\infty}|_{\mathcal{R}}^{-1}(X_0^{sing})=E \subset \mathcal{R}$ is an analytic subset, and since analytic subsets have real codimension 2, it follows that $X_{\infty}\setminus X_0^{reg} \subset X_{\infty}$ has Haussdorff codimension at least 2, and by \cite[Theorem 3.7]{CC2}, we have $\overline{(X_0^{reg}, g_{\infty})} = X_{\infty}$. 
	
	Now we show that $\pi_{\infty}|_{\mathcal{R}}$ is holomorphic. Consider the holomorphic maps $\pi: (X, \omega_{\varphi_t}, p)\to X_0\subset  (\C P^N, \omega_{FS}, p)$, since $(X, \omega_{\varphi_t}, p)$ Gromov-Haussdorff converge to $X_{\infty}$, by Cheeger-Colding theory \cite{CC1}, for any $K\subset \subset \mathcal{R}$ containing $p$, there exist maps $\iota_{t_i}:K\to (X, \omega_{\varphi_{t_i}})$ such that $\iota_{t_i}^{\star}g_{t_i}\to g_{\infty}$ and $\iota_{t_i}^{\star}J\to J_{\infty}$ in the smooth topology, and we also get a sequence of holomorphic maps $\pi_i = \pi\circ\iota_{t_i}:(K, \iota_{t_i}^{\star}g_{t_i}, \iota_{t_i}^{\star}J)\to X_0\subset \C P^N$. Furthermore, if we regard these maps as harmonic maps, then we have \[|d\pi_i|^2_{\omega_{\varphi_t, \omega_{FS}}} = \Tr_{\omega_{\varphi_t}}\pi_i^{\star}\omega_{FS}\leq C\]
	hence by the regularity theory of harmonic maps (\cite{Sch}), we have uniform $C^{\infty}$ estimates on the maps $\|\pi_i^{l}\|_{C^{k, \alpha}}(K)\leq C_K$, for some constant $C_K$ independent of $i$, which allows us to extract a limit of the maps $\pi_i:K\to \C P^N$ to a map $\pi_{\infty}:K\to \C P^N$ and since the convergence of the maps are smooth, and the convergence of the metrics $\iota_{t_i}^{\star}g_{t_i}\to g_{\infty}$ and the complex structures $\iota_{t_i}^{\star}J\to J_{\infty}$ are all smooth, it follows that the holomorphicity of the maps $\pi_i$ passes to the limit, and hence the map $\pi_{\infty}$ is holomorphic. 
\end{proof}

\begin{Proposition}
	In fact we have $\mathcal{R} = \iota_{\infty}(X_0^{reg})$. 
\end{Proposition}
\begin{proof}
	The proof is the same as in \cite[Lemma 2.2]{RZ13}. 
\end{proof}

\subsection{Identification of $X_0$ with the geometry of singular Calabi-Yau}
In this section, we identify the geometry of the singular Calabi-Yau current $X_{\infty} = \overline{(X_0^{reg}, g_{\infty})}$ with the variety $X_0$ itself. This result is the analogue of the result in \cite{JS14}, where the similar thing was shown in the compact case, our proof follows the approach in \cite{JS14}, adapted to the non-compact case. The idea is based on ideas developed in \cite{DS14} together with a new gradient estimate for the potential $\varphi_t$ with respect to the Calabi-Yau metrics $\omega_{\varphi_t}$. 

\subsubsection{A gradient bound for $\varphi_{0}$}
The goal of this section is to prove the following estimate
\begin{Proposition}\label{phi-grad-bdd}
	The following bound hold
	\begin{equation*}
		\sup_{\pi^{-1}(X_0^{reg})}|\nabla_{\omega_{\varphi_0}}\varphi_{0}|\leq C
	\end{equation*}
\end{Proposition}

\begin{Proposition}\label{v-bdd}
	If we set let $v_t = \varphi_t-t\dot{\varphi}_t$, then we have a uniform estimate
	\begin{equation*}
	\sup_X |v_t|\leq C
	\end{equation*}
\end{Proposition}
\begin{proof}
	Recall from the construction of $\hat{\omega}_t$ (Proposition~\ref{prop: bkground-metric-improvement}) that 
	\begin{align*}
		\hat{\omega}_t &= \omega_t+i\d\db u_t\\
		& = (1-t)\omega_0+t\omega_1 +i\d\db u_t
	\end{align*}
	where $\omega_0 = \pi^{\star}\omega_{X_0}$ and $\omega_{X_0}$ is a K\"ahler metric on $X_0$. So we have
	\begin{align*}
	\lapl_{\varphi_t}\varphi_t &= n-\Tr_{\varphi_t}\hat{\omega}_t\\
	& = n-(1-t)\Tr_{\varphi_t}\omega_0-t\Tr_{\varphi_t}\omega_1-\lapl_{\varphi_t}u_t. 
	\end{align*}
	Recall that by the construction of $\hat{\omega}_t$, Proposition~\ref{bkground-metric 1}, we have 
	\[\log \frac{(\hat{\omega}_t+i\d\db\varphi_t)^n}{\hat{\omega}_t^n} = f_t \in C^{\infty}_{-\gamma-2}. \]
	for some $0<\gamma < 2n-2$. Differentiating the equation, we have 
	\begin{equation}\label{eq: lapl-varphi-dot-1}
	\lapl_{\varphi_t}\dot{\varphi_t} = \dot{f_t}-\Tr_{\varphi_t}\dd{t}\hat{\omega}_t +\Tr_{\hat{\omega}_t}\dd{t}\hat{\omega}_t \in C^{\infty}_{-\gamma-2}(X)
	\end{equation}
	so we have $\dot{\varphi}\in C^{\infty}_{-\gamma}(X)$ for $t>0$. 
	
	If we differentiate the equation $(\hat{\omega}_t+i\d\db\varphi_t)^n = i^{n^2}\Omega\wedge\bar{\Omega}$ with respect to $t$, we obtain another expression for $\lapl_{\varphi_t}\dot{\varphi}_t$
	\begin{equation}\label{eq: lapl-varphi-dot-2}
	\lapl_{\varphi_t}\dot{\varphi_t} = -\lapl_{\varphi_t}\dot{u_t}+\Tr_{\varphi_t}(\omega_0-\omega_1)
	\end{equation}
	The equations \eqref{eq: lapl-varphi-dot-1} and \eqref{eq: lapl-varphi-dot-2} imply that $v_t$ satisfy the two equations
	\begin{equation}\label{eq: lapl-v_t-1}
	\lapl_{\varphi_t}v_t = n-\Tr_{\varphi_t}\omega_0 -\lapl_{\varphi_t}(u_t-t\dot{u_t})
	\end{equation}
	and
	\begin{equation}\label{eq: lapl-v_t-2}
	\lapl_{\varphi_t}v_t = n-\Tr_{\varphi_t}\hat{\omega}_t-t(\dot{f_t}-\Tr_{\varphi_t}\dd{t}\hat{\omega}_t 
	+\Tr_{\hat{\omega}_t}\dd{t}\hat{\omega}_t)
	\end{equation}
	From the first equation and Proposition~\ref{Schwartz Lemma}, we see that $|\lapl_{\varphi_t}\dot{\varphi_t}|\leq C$ uniformly in $t$. From the second equation we see that $|\lapl_{\varphi_t}v_t|\leq Cr^{-\gamma-2}$ away from a compact set $K$, so we have a uniform bound $|\lapl_{\varphi_t}v_t|_{L^p}\leq C$ for $p>\frac{2n}{\gamma+2}$. 
	Since $v_t\in C^{\infty}_{-\gamma}$, we can do integrate by parts to get 
	\begin{align*}
	-\int_X |v_t|^{p-2}v_t i\d\db v_t\wedge \omega_{\varphi_t}^{n-1} & = \lim_{R\to \infty}(p-1)\int_{B_R}|v_t|^{p-2}i\d v_t\wedge\db v_t\wedge \omega_{\varphi_t}^{n-1} \\
	&\qquad -\lim_{R\to \infty}\left(\int_{\d B_R}|v_t|^{p-2}v_t i\db v_t \wedge \omega_{\varphi_t}^{n-1} \right)\\
	& = \frac{4(p-1)}{p^2}\int_{X}i\d |v_t|^{\frac{p}{2}}\wedge\db |v_t|^{\frac{p}{2}}\wedge \omega_{\varphi_t}^{n-1} 
	\end{align*}
	the boundary term goes away when $p>\frac{2n-2}{\gamma}$ since $|\nabla^k v_t| = O(r^{-\gamma-k})$. Hence we get 
	\begin{equation*}
	\int_X |\d |v_t|^{\frac{p}{2}}|^2\omega_{\varphi_t}^n = -\frac{np^2}{4(p-1)}\int_X|v_t|^{p-2}v_t(\lapl_{\varphi_t}v_t)\omega_{\varphi_t}^n
	\end{equation*}
	combined with the Sobolev inequality, one gets
	\begin{equation*}
	\left(\int_X|v_t|^{p\frac{n}{n-1}}i^{n^2}\Omega\wedge\bar{\Omega}\right)^{\frac{n-1}{n}}\leq C\frac{np^2}{p-1}\int_X|v_t|^{p-1}|\lapl_{\varphi_t}v_t| i^{n^2}\Omega\wedge\bar{\Omega}
	\end{equation*}
	applying Holder, we get
	\begin{equation*}
	\|v_t\|_{L^{p\frac{n}{n-1}}}\leq C\frac{np^2}{p-1}\|\lapl_{\varphi_t}v_t\|_{L^{\frac{np}{n+p+1}}}
	\end{equation*}
	hence for $p>\frac{2n}{\gamma}$, we have
	\begin{equation}\label{eq: v_t-lp-bdd}
	\|v_t\|_{L^p}\leq C_p
	\end{equation}
	where $C_p$ depends on $\|\lapl_{\varphi_t}v_t\|_{L^{\frac{np}{n+p}}}$. 
	and also for $p>\frac{2n-2}{\gamma}$
	\begin{equation*}
	\|v_t\|_{p\frac{n}{n-1}}^p\leq C\frac{np^2}{p-1}\|v_t\|_{L^p}^{p-1}\|\lapl_{\varphi_t}v_t\|_{L^p}
	\end{equation*}
	we can then apply Moser iteration to this to get the estimate
	\begin{equation*}
	\|v_t\|_{L^{\infty}}\leq B_p\|v_t\|_{L^p}\leq B_p C_p
	\end{equation*}
	where $C_p$ is the constant from \eqref{eq: v_t-lp-bdd} and $B_p$ depends only on the $L^p$ norm of $\|\lapl_{\varphi_t}v_t\|_{L^p}$. 
\end{proof}

\begin{Corollary}
	For any compact set $K\subset \subset \pi^{-1}(X_0^{reg})$, we have an estimate 
	\begin{equation*}
		|v_t|_{C^{k, \alpha}(K)}\leq C(K, k, \alpha)
	\end{equation*}
	uniformly in $t$ as $t\to 0$. 
\end{Corollary}

\begin{proof}
	This follows from the equation \eqref{eq: lapl-v_t-1} and the fact that $\omega_{\varphi_t}$ and the right hand side of the equation is uniformly bounded in $C^{\infty}_{loc}(\pi^{-1}(X_0^{reg}))$. 
\end{proof}

\begin{Proposition}\label{grad-v-bdd}
	We also have the following local uniform gradient estimate for $v_t$. 
	\begin{equation*}
	\sup_K|\nabla_t v_t|\leq C_K
	\end{equation*}
	for any $K\subset \subset X$. 
\end{Proposition}
\begin{proof}
	By the Bochner formula, we have
	\begin{align*}
	\lapl_{\varphi_t}|\nabla v_t|_{g_{\varphi_t}}^2 &= |\nabla\nabla v_t|_{g_{\varphi_t}}^2+|\d\db v_t|_{g_{\varphi_t}}^2 - 2\Re(\nabla v_t\cdot \nabla \Tr_{\varphi_t}\omega_0)- 2\Re(\nabla v_t\cdot \nabla \lapl_{\varphi_t}(u_t-t\dot{u_t}))\\
	&\geq -2|\nabla v_t|_{g_{\varphi_t}}^2-|\nabla \Tr_{\omega_{\varphi_t}}\omega_{0}|_{g_{\varphi_t}}^2 - |\nabla \lapl_{\varphi_t}(u_t-t\dot{u_t})|^2_{g_{\varphi_t}}
	\end{align*}
	we also have from \eqref{Schwartz inequality},
	\begin{equation*}
		\lapl_{\omega_{\varphi_t}}\Tr_{\omega_{\varphi_t}}\omega_0 \geq -C + c_0|\nabla \Tr_{\omega_{\varphi_t}}\omega_0|^2
	\end{equation*}
	 If we set $H_t = |\nabla v_t|_{g_{\varphi_t}}^2+A\Tr_{\omega_{\varphi_t}}\omega_0 $, then $H_t\geq 0$ and satisfies
	\begin{equation*}
		\lapl_{\varphi_t}H_t \geq -H_t-C
	\end{equation*}
	We can apply Moser iteration to this, since $\omega_{\varphi_t}$ has uniform Ricci bounds and volume lower bound, this then gives the estimate
	\begin{equation}\label{L^2 to L^inf}
	\|H_t\|_{L^{\infty}(B_{g_{\infty}, R}(p))}\leq C\|H_t\|_{L^2(B_{g_{\infty}, 2R}(p))}
	\end{equation}
	for $R$ sufficiently large. Note that for $R$ sufficiently large $\omega_{\varphi_t}$ converge uniformly in $C^{\infty}$ to $g_{\infty}$ on the region $B_{g_{\infty}, 2R}(p)\setminus B_{g_{\infty}, R}(p)$, hence we can also choose cutoff functions with uniformly controlled gradients and standard Moser iteration gives the inequality. Now it suffices to show that $\|H_t\|_{L^2(B_{g_{\infty}, 2R})}$ is bounded. 
	\begin{align*}
	\int_{B_{2R}}|H_t|^2&\leq \|H_t\|_{L^{\infty}(B_{2R})}\int_{B_{2R}}|H_t|\\
	&\leq C\|H_t\|_{L^2(B_{4R})}\|H_t\|_{L^1(B_{2R})}\\
	&\leq C(\|H_t\|_{L^2(B_{2R})}+\|H_t\|_{L^2(B_{4R}\setminus B_{2R})})\|H_t\|_{L^1(B_{2R})}
	\end{align*}
	if $R$ is sufficiently large, then $B_{4R}\setminus B_{2R}$ doesn't contain any of $\pi^{-1}(X_0^{sing})$, hence $\|H_t\|_{L^2(B_{4R}\setminus B_{2R})}$ is uniformly bounded in $t$ on $B_{4R}\setminus B_{2R}$ by the Corollary above. So we have
	\begin{equation*}
	\|H_t\|^2_{L^2(B_{2R})}\leq C(\|H_t\|_{L^2(B_{2R})}+C)\|H_t\|_{L^1(B_{2R})}
	\end{equation*}
	hence either $\|H_t\|_{L^2(B_{2R})}$ is bounded by 1 and we are done, or we get the bound 
	\begin{equation}\label{L^1 to L^2}
	\|H_t\|_{L^2(B_{2R})}\leq C \|H_t\|_{L^1(B_{2R})}
	\end{equation}
	so it suffices to prove an $L^1$ bound for $H_t$ on compact sets. 
	
	Choose cutoff function $\eta$ such that $\eta = 1$ on $B_{2R}$ for all $t$, then
	\begin{align*}
	\int_{B_{2R}}H_t &\leq \int_{X} \eta^2 H_t\\
	& \leq \int_X \eta^2|\nabla v_t|^2\omega_{\varphi_t}^n +C\\
	&\leq  -\int_X\eta^2v_t(\lapl_{\varphi_t}v_t) + 2\int_X \eta|\nabla \eta||\tilde{v_t}||\nabla \tilde{v_t}|+C
	\end{align*}
	and so we have 
	\[\int_{B_{2R}}H_t\leq C\int_{X}(\eta^2+|\nabla \eta|^2)v_t^2\leq C\]
	which gives us the $L^1$ bound, combined with \eqref{L^1 to L^2} and \eqref{L^2 to L^inf}, we get
	\begin{equation*}
	\|H_t\|_{L^{\infty}(B_{R})}\leq C
	\end{equation*}
	as desired. 
\end{proof}

\begin{Proposition}\label{phidot-ptwise-bdd}
	For any $x\in \iota_{\infty}(X_0^{reg})$, we have a bound 
	\begin{equation*}
		|\dot{\varphi_t}(x)|\leq C
	\end{equation*}
	for some constant $C$ potentially depending on the point $x$. 
\end{Proposition}

\begin{proof}
	Fix $x\in \iota_{\infty}(X_0^{reg})$, then fix a ball $B_{g_{\infty}, \eps}(x)\subset \iota_{\infty}(X_0^{reg})$ on which the metrics $g_{\varphi_t}$ converge smoothly to $g_{\infty}$, also fix a set $K\subset \subset X$ containing all of $\pi_{\infty}^{-1}(X_0^{sing})$ and also $B_{g_{\infty}, \eps}(x)$. Then by the Green's formula representation formula, we have
	\begin{align*}
		\dot{\varphi_t}(x) &= -\int_{X}\lapl_{\varphi_t}\dot{\varphi_t}(y)G_t(x, y)\omega_{\varphi_t}^n(y)\\
		& = -\int_{B_{g_{\infty}, \eps}(x)}\lapl_{\varphi_t}\dot{\varphi_t}(y)G_t(x, y)\omega_{\varphi_t}^n(y)-\int_{K\setminus B_{g_{\infty}, \eps}(x)}\lapl_{\varphi_t}\dot{\varphi_t}(y)G_t(x, y)\omega_{\varphi_t}^n(y)\\
		&\qquad -\int_{X\setminus K}\lapl_{\varphi_t}\dot{\varphi_t}(y)G_t(x, y)\omega_{\varphi_t}^n(y)
	\end{align*}
	where $G_t(x, y)$ is the positive decaying Green's function on $(X, \omega_{\varphi_t})$. By the estimates for Green's function \cite[p.190]{MSY}, \cite{LiTam, LiYau}, the Greens functions $G_t(x, y)$ satisfy the uniform estimates
	\begin{equation*}
		C^{-1}d_{t}(x, y)^{2-2n}\leq G_t(x, y)\leq Cd_{t}(x, y)^{2-2n}
	\end{equation*}
	where $d_t$ is the distance function induced by $g_{\varphi_t}$. 
	And since $\lapl_{\varphi_t}\dot{\varphi_t} = -\lapl_{\varphi_t}\dot{u_t}+\Tr_{\varphi_t}(\omega_0-\omega_1)$, this implies $|\lapl_{\varphi_t}\dot{\varphi_t}| \leq |\lapl_{\varphi_t}\dot{u_t}|+\Tr_{\varphi_t}(\omega_0+\omega_1)\leq C+\Tr_{\varphi_t}\omega_1$ and we have
	\begin{align}
\notag		|\dot{\varphi_t}(x)| &\leq \int_{B_{g_{\infty}, \eps}(x)}|\lapl_{\varphi_t}\dot{\varphi_t}|(y)d_{t}(x, y)^{2-2n}\omega_{\varphi_t}^n(y) + \int_{K\setminus B_{g_{\infty}, \eps}(x)}|\lapl_{\varphi_t}\dot{\varphi_t}|(y)d_{t}(x, y)^{2-2n}\omega_{\varphi_t}^n(y)\\
\label{eqn:107}		&\qquad +\int_{X\setminus K}|\lapl_{\varphi_t}\dot{\varphi_t}|(y)d_{t}(x, y)^{2-2n}\omega_{\varphi_t}^n(y)
	\end{align}
	and we analyze the three terms in the above formula seperately. For the first term, we note that $\lapl_{\varphi_t}\dot{\varphi_t}$ is uniformly bounded on $B_{g_{\infty}, \eps}(x)$, so
	\begin{equation*}
		 \int_{B_{g_{\infty}, \eps}(x)}|\lapl_{\varphi_t}\dot{\varphi_t}|(y)d_{\varphi_t}(x, y)^{2-2n}\omega_{\varphi_t}^n(y) \leq C\int_{B_{g_{\infty}, \eps}(x)}d_{t}(x, y)^{2-2n}\leq C
	\end{equation*}
	For the second term, observe that on $K\setminus B_{g_{\infty}, \eps}(x)$, $d_t(x, y)^{2-2n}$ is bounded by $C\eps^{2-2n}$, so
	\begin{equation*}
		\int_{K\setminus B_{g_{\infty}, \eps}(x)}|\lapl_{\varphi_t}\dot{\varphi_t}|(y)d_{\varphi_t}(x, y)^{2-2n}\omega_{\varphi_t}^n(y) \leq C\left(1+\int_{K\setminus B_{g_{\infty}, \eps}(x)}\Tr_{\varphi_t}\omega_1\right)
	\end{equation*}
	hence it suffices to bound the integral of $\Tr_{\varphi_t}\omega_1$, to do this, we integrate by parts
	\begin{align*}
		\int_{K}\omega_1\wedge\omega_{\varphi_t}^{n-1}  &= \int_{K}\omega_1\wedge(\hat{\omega}_t+i\d\db\varphi_t)^{n-1} = \int_{K}\omega_1\wedge\hat{\omega}_t^{n-1}\\
		&\qquad +\int_{\d K} \d\varphi_t\wedge \omega_1\wedge \left(\sum_{l=0}^{n-2}\binom{n-1}{l}\hat{\omega}_t^l\wedge(i\d\db\varphi_t)^{n-2-l}\right)\\
		&\leq C
	\end{align*}
	because $\varphi_t$ and its derivatives are all bounded on the boundary of $K$. \\
	The last term in \eqref{eqn:107} is bounded because $|\lapl_{\varphi_t}\dot{\varphi_t}|\leq Cd_t(x, y)^{-2-\beta}$ on $X\setminus K$, so we have
	\begin{equation*}
		 \int_{X\setminus K}|\lapl_{\varphi_t}\dot{\varphi_t}|(y)d_{\varphi_t}(x, y)^{2-2n}\omega_{\varphi_t}^n(y) \leq C\int_{X\setminus K}d_t(x, y)^{-2n-\beta} \leq C
	\end{equation*} 
	and we get our result. 
\end{proof}
\begin{proof}[proof of Proposition \ref{phi-grad-bdd}]
	Note that we already know $|\nabla_{g_{\infty}}\varphi_0|$ is bounded and decaying at infinity, so it suffices to prove that it's bounded near $\pi_{\infty}^{-1}(X_0^{sing})$. Fix a compact set $K$ containing $\pi_{\infty}^{-1}(X_0^{sing})$, then by Proposition \ref{grad-v-bdd}, $|\nabla v_t|\leq C$, but on $X_0^{reg}$, $v_t$ converges to $\varphi_0$ smoothly on compact sets, hence we get our result. 
\end{proof}

The main goal of this gradient bound is to show the following. 
\begin{Proposition}
	For any holomorphic section $s\in H^0(X_0, L^k)$ satisfies
	\begin{equation*}
		\sup_K|s|_{h^k_{\infty}}\leq C
	\end{equation*}
	and 
	\begin{equation*}
		\sup_K |\nabla s|_{h^k_{\infty}, k\omega_{\varphi_0}}\leq C
	\end{equation*}
\end{Proposition}
\begin{proof}
	Locally we can write $h_{\infty} = e^{-\varphi_0}\hat{h}_0$ where $-i\d\db \log \hat{h}_0 = \hat{\omega}_0$, since $\hat{\omega}_0 = \pi^{\star}\omega_{FS}$ on $K$, we have simply $h_{\infty} = e^{-\varphi_0}h_{FS}$, and by the $C^0$ bound for $\varphi_0$, it follows that $|s|_{h^k_{\infty}}\leq C|s|_{h^k_{FS}}\leq C$. To see the bound for the gradient, we note
	\[|\nabla_{h^k_{\infty}} s|^2_{h^k_{\infty}, k\omega_{\varphi_0}} = |\nabla_{h^k_{FS}}s+ k (\d\varphi_0)s|_{h^k_{\infty}, k\omega_{\varphi_0}}\leq |\nabla_{h^k_{FS}}s|_{h^k_{FS}, k\omega_{\varphi_0}}+ k |\nabla\varphi_0|_{k\omega_{\varphi_0}}|s|_{h^k_{\infty}}\]
	and by the gradient estimate \eqref{phi-grad-bdd} $|\nabla \varphi_0|_{k\omega_{\varphi_0}}\leq C$, so the second term is bounded, and by the estimate \eqref{lower-bdd}, we have
	$|\nabla_{h^k_{FS}}s|_{h^k_{FS}, k\omega_{\varphi_0}}\leq C|\nabla_{h^k_{FS}}s|_{h^k_{FS}, k\omega_{FS}}\leq C$ and we get the bound that we wanted. 
\end{proof}

We will need the boundedness of $|s|_{h^k_{\infty}}$ and $|\nabla s|_{h^k_{\infty}, k\omega_{\varphi_0}}$ to make the Moser iteration argument work with cutoff functions in the next section. 

\subsubsection{$L^2$ estimates on $X_0$}
The argument of this section follows in the same way as in \cite{JS14}, with minor modifications. 

We first quote a proposition stating the existence of good cutoff functions on $X_{\infty}$ from \cite{DS14}.

\begin{Lemma}\cite[Proposition 3.5]{DS14}
	There exist cutoff functions $\rho_{\eps}$ on $X_{\infty}$ satisfying the following
	\begin{enumerate}
		\item $0\leq \rho_{\eps}\leq 1$
		\item $\text{supp}(\rho_{\eps})\subset \subset \mathcal{R} = X_0^{reg}$
		\item For any compact set $K\subset \subset \mathcal{R}$, there exist $\eps_K>0$ such that for all $\eps<\eps_K$, we have $\rho_{\eps} = 1$ on $K$. 
		\item $\int_{X}|\nabla \rho_{\eps}|^2\to 0$ as $\eps\to 0$. 
	\end{enumerate}
\end{Lemma}

We recall the following version of Hormander's $L^2$ estimates for the $\db$ equation. 

\begin{Theorem}\cite[Cor 5.3]{Dem}\label{L2 est}
	Let $(M, \omega)$ be a K\"ahler manifold. Assume $M$ is weakly pseudoconvex. Let $(L, h)$ be a Hermitian line bundle with curvature with (possibly) singular Hermitian metric $h$, and suppose
	\begin{equation*}
	-i\d\db\log h+\Ric(\omega) \geq \gamma(x)\omega
	\end{equation*}
	then for any $\beta\in \bigwedge^{0, 1}\otimes L$, with $\db \beta = 0$, there exist a section $s\in L$ satisfying $\db s = \beta$ with 
	\begin{equation*}
	\int_{M}|s|^2_{h}\omega^n\leq \int_{M} \frac{1}{\gamma}|\beta|^2_{h, \omega}\omega^n,
	\end{equation*}
provided the integral on the RHS is finite.
\end{Theorem}

Now we will prove a version of the above theorem on $X$ equipped with a singular metric $\omega_{\varphi_0}$ that we constructed. First we fix a Hermitian metric $h_0$ on $L$ such that $-i\d\db\log h_0 = \hat{\omega}_0$, which is possibly by the $\d\db$-Lemma. 

\begin{Theorem}\label{L2 est on X_0}
	Let $h_{\infty} = e^{-\varphi_0}h_0$, so $-i\d\db \log h_{\infty} = \omega_{\varphi_0}$, and $K\subset\subset X$ a compact subset with pseudoconvex boundary. 
	Then for any  $\beta\in \bigwedge^{0, 1}\otimes L^k$, with compact support and $\text{supp}(\beta)\subset X_0^{reg}\cap K$ and $\db \beta = 0$, there exist a section $u\in H^0(L^k)$ satisfying $\db u = \beta$ with 
	\begin{equation*}
	\int_{K}|u|^2_{h^k_{\infty}}\omega_{\varphi_0}^n\leq \int_{K} |\beta|^2_{h^k_{\infty}, k\omega_{\varphi_0}}\omega_{\varphi_0}^n
	\end{equation*}
\end{Theorem}
\begin{proof}
	By Assumption~\ref{ass: mainAss}, we know that $\hat{\omega}_0+i\d\db\psi_{\eps}\geq  \eps\omega$, which implies $\hat{\omega}_0+ti\d\db\psi_{\eps}\geq (1-t)\hat{\omega}_0+t\eps\omega$. By the discussion in the previous sections, we can solve
	\[\omega_{\varphi_t}^n  = \left((1-t)\hat{\omega}_0+t\eps\omega +i\d\db\tilde{\varphi}_{t}\right)^n = i^{n^2}\Omega\wedge\bar{\Omega}\]
	with $\tilde{\varphi}_t$ is bounded on any compact set $K\subset \subset X$, uniformly as $t\to 0$ and $\tilde{\varphi}_t \to \varphi_0$ in $L^{\infty}_{loc}(X)$ and in $C^{\infty}_{loc}(X^{reg}_0)$. 
	We pick a metric $h_0$ on $L$ such that $-i\d\db\log h_0 = \hat{\omega}_0$, then if we set $\tilde{h}_t = e^{-t\psi_{\eps}-\tilde{\varphi}_t}h_0$, it satisfies
	\begin{equation*}
		-i\d\db\log \tilde{h}_t^k  = k(\hat{\omega}_0+ti\d\db\psi_{\eps}+ i\d\db\tilde{\varphi}_t) \geq k\omega_{\varphi_t}
	\end{equation*}
	By the previous lemma, we can always solve $\db u_t = \beta$, satisfying the estimate
	\begin{equation*}
		\int_K |u_t|^2_{\tilde{h}_t^k}\omega_{\varphi_t}^n \leq\int_K |\beta|^2_{\tilde{h}_t^k, k\omega_{\varphi_t}}\omega_{\varphi_t}^n =\int_K e^{-tk\psi_{\eps}-k\tilde{\varphi}_t}|\beta|^2_{h^k_0, k\omega_{\varphi_t}}\omega_{\varphi_t}^n
	\end{equation*}
	Since $\beta$ is compactly supported on $X_0^{reg}$,  $\omega_{\varphi_t}\to \omega_{\varphi_0}$ on the support of $\beta$, and $e^{-tk\psi_{\eps}}\to 1$ in $L^1_{loc}$, so we have
	\[\lim_{t\to 0}\int_K |\beta|^2_{\tilde{h}^k_t, k\omega_{\varphi_t}}\omega_{\varphi_t}^n = \int_K e^{- k \varphi_0}|\beta|^2_{h_0^k, k\omega_{\varphi_0}}\omega_{\varphi_0}^n\]
	and since $e^{-tk\psi_{\eps}- k\tilde{\varphi}_t}$ is bounded from below on any compact set $K$, it follows that 
	\[\int_{K}|u_t|_{h_0^k}^2i^{n^2}\Omega\wedge\bar{\Omega} \leq C\int_Ke^{-tk\psi_{\eps}- k\tilde{\varphi}_t}|u_t|^2_{h_0^k}i^{n^2}\Omega\wedge\bar{\Omega} = C\int_{K}|u_t|_{\tilde{h}_t^k}^2\omega_{\varphi_t}^n\leq C\]
	hence there exist a weakly convergent subsequence $u_t \rightharpoonup u$ in $L^2(K, h^k_0)$ and the equation $\db u_t = \beta$ carries through the limit in the weak convergence, so we have $\db u = \beta$. Since the sections $u_t - u$ are holomorphic and weakly converge to $0$, it follows that the convergence is smooth it happens strongly, hence we have
	\begin{equation*}
		\int_Ke^{- k\varphi_0}|u|^2_{h^k_0}i^{n^2}\Omega\wedge\bar{\Omega}\leq \int_Ke
		^{- k \varphi_0}|\beta|^2_{h^k_0, \omega_{\varphi_0}} i^{n^2}\Omega\wedge\bar{\Omega}
	\end{equation*}
\end{proof}

\begin{Proposition}\label{Sobolev}
	The following Sobolev inequality hold for $f\in L^{\infty}\cap H^{1}(X_0^{reg}, \omega_{\infty})$
	\begin{equation*}
	\left(\int_{X_0^{reg}} |f|^{2\frac{n}{n-1}}\omega_{\infty}^n\right)^{\frac{n-1}{n}}\leq C \int_{X_0^{reg}}|\nabla f|_{g_{\infty}}^2\omega_{\infty}^n
	\end{equation*}
\end{Proposition}
\begin{proof}
	Without loss of generality, we can assume $f\geq 0$. If $f$ is supported in $X_0^{reg}$, this follows from \cite{Cr}. For $f\in L^{\infty}$, we can define $f_{\eps} = f\rho_{\eps}$, $f_{\eps}$ is supported in $X_0^{reg}$, then we clearly have $\|f_{\eps}\|_{L^2}\to \|f\|_{L^2}$, and we also have
	\begin{equation*}
		\int_{X}|\nabla f_{\eps}|^2 = \int_X\rho_{\eps}^2|\nabla f|^2+\int_Xf^2|\nabla \rho_{\eps}|^2+2\int_Xf\rho_{\eps}\langle\nabla f, \nabla \rho_{\eps}\rangle
	\end{equation*}
	the second and third term goes to $0$ as $\eps\to 0$ because $\int_{X}|\nabla \rho_{\eps}|^2\to 0$, and this gives what we wanted. 
\end{proof}

\begin{Lemma}\label{moser-iteration}
	Suppose $u\geq 0$ is a bounded function on $X_0^{reg}$ that satisfy
	\begin{equation*}
	\lapl_{\omega_\infty}u\geq -Au
	\end{equation*}
	then for $R\geq 1$ sufficiently large (so that $X_0^{sing}\subset B_R(p)$), we have the estimate
	\begin{equation*}
	\|u\|_{L^{\infty}(B_R(p))}\leq C(A+CR^{-2})^{\frac{n}{2}}\|u\|_{L^2(B_{2R}(p))}
	\end{equation*}
\end{Lemma}

\begin{proof}
	Using the Sobolev inequality above and the cutoff function, we can do Moser iteration on $(X_0^{reg}, g_{\infty})$
	\begin{align*}
	A\int_X \eta^2\rho_{\eps}^2u^{p+1}\omega_{\infty}^n&\geq \int_X \eta^2\rho_{\eps}^2u^{p}(-\lapl u)\omega_{\infty}^n\\
	& = \frac{4p}{(p+1)^2}\int_X\eta^2\rho_{\eps}^2|\nabla u^{\frac{p+1}{2}}|^2\omega_{\infty}^n+2\int_X\rho_{\eps}^2\eta(\nabla\eta\cdot\nabla u)u^p\omega_{\infty}^n\\
	&\qquad + \frac{4}{(p+1)}\int_X \eta^2 \rho_{\eps}(\nabla\rho_{\eps}\cdot\nabla u^{\frac{p+1}{2}})u^{\frac{p+1}{2}}\omega_{\infty}^n\\
	& \geq \frac{4p}{(p+1)^2}\int_X\eta^2\rho_{\eps}^2|\nabla u^{\frac{p+1}{2}}|^2\omega_{\infty}^n+2\int_X\rho_{\eps}^2\eta(\nabla\eta\cdot\nabla u)u^p\omega_{\infty}^n\\
	&\qquad - \frac{4}{(p+1)}\left(\int_X\eta^2\rho_{\eps}^2|\nabla u^{\frac{p+1}{2}}|^2\omega_{\infty}^n\right)^{\frac{1}{2}}\left(\int_X \eta^2|\nabla\rho_{\eps}|^2u^{p+1}\omega_{\infty}^n\right)^{\frac{1}{2}}
	\end{align*}
	when $u$ is bounded, we can take a limit as $\eps$ goes to $0$ and the last term will disappear, so we have
	\begin{align*}
	A\int_X \eta^2u^{p+1}\omega_{\infty}^n&\geq \frac{4p}{(p+1)^2}\int_X\eta^2|\nabla u^{\frac{p+1}{2}}|^2\omega_{\infty}^n+\frac{4}{p+1}\int_X\eta(\nabla\eta\cdot \nabla u^{\frac{p+1}{2}})u^{\frac{p+1}{2}}\omega_{\infty}^n\\
	&\geq \frac{3p}{(p+1)^2}\int_X\eta^2|\nabla u^{\frac{p+1}{2}}|^2\omega_{\infty}^n - \frac{16}{p}\int_X|\nabla \eta|^2u^{p+1}\omega_{\infty}^n
	\end{align*}
	which implies
	\[
	\int_X|\nabla \eta u^{\frac{p+1}{2}}|^2\omega_{\infty}^n\leq \frac{(p+1)^2}{p}\int_X(A\eta^2+\frac{17}{p}|\nabla \eta|^2)u^{p+1}\omega_{\infty}^n
	\]
	then by the Sobolev inquality from Proposition \ref{Sobolev}, we have for any $p>0$, 
	\[
	\left(\int_X|\eta u|^{(p+1)\frac{n}{n-1}}\omega_{\infty}^n\right)^{\frac{n-1}{n}}\leq \frac{C(p+1)^2}{p}\int_X(A\eta^2+\frac{17}{p}|\nabla \eta|^2)u^{p+1}\omega_{\infty}^n
	\]
	by carefully choosing cutoff functions $0\leq \eta_k\leq 1$ such that $\text{supp}(\eta_k) \subset B_{(1+2^{-k})R}$, $\eta_k = 1$ on $B_{(1+2^{-k-1})R}$ and $|\nabla \eta_k|\leq CR^{-1}2^k$, and set $p_k = 2(\frac{n}{n-1})^k$, then for $k = 0, 1, 2, \ldots$ we have
	\[
	\|u\|_{L^{p_{k+1}}(B_{(1+2^{-k-1})R})}^{p_k}\leq C(Ap_k+CR^{-2}4^{k})\|u\|_{L^{p_k}(B_{(1+2^{-k})R})}^{p_k}
	\]
	iterating gives 
	\[
	\sup_{B_R}u\leq C_{sob}^{\frac{n}{2}}C(2A+CR^{-2})^{\frac{n}{2}}\|u\|_{L^2(B_{2R})}
	\]
\end{proof}

We now prove $L^2$ estimates for holomorphic sections of $L^k$. 

\begin{Proposition}\label{grad-bdd-for-section}
	If $s$ is a holomorphic section of $(L^k, h^k_{\infty})$, then the following estimates hold on $(X_0^{reg}, kg_{\infty})$ for $R$ large enough so that $B_R(p)$ contains all of $X_0^{sing}$, 
	\begin{equation*}
	\sup_{B_R(p)}|s|_{h^k_{\infty}}\leq C \|s\|_{L^2_{h^k_{\infty}, kg_{\infty}}(B_{2R}(p))}
	\end{equation*}
	\begin{equation*}
	\sup_{B_R(p)}|\nabla s|_{h^k_{\infty}, kg_{\infty}} \leq C\|s\|_{L^2_{h^k_{\infty}, kg_{\infty}}(B_{2R}(p))}
	\end{equation*}
\end{Proposition}
\begin{proof}
	For a holomorphic section $s$, we have $\nabla_{\bj}s = 0$, so $g^{i\bj}\nabla_{\bj}\nabla_is = -ns$. It follows then from standard calculations that
	\begin{equation*}
	\lapl|s|\geq -n|s|
	\end{equation*}
	and 
	\begin{equation*}
	\lapl|\nabla s|\geq -(n+2)|\nabla s|
	\end{equation*}
	so now we can apply Lemma \ref{moser-iteration} with $u = |s|$ and $u = |\nabla s|$ to get
	\begin{equation*}
	\|s\|_{L^{\infty}(B_R)}\leq C\|s\|_{L^2(B_{2R})}
	\end{equation*}
	and
	\begin{equation}\label{gradient bound for sections}
	\|\nabla s\|_{L^{\infty}(B_R)}\leq C\|\nabla s\|_{L^2(B_{2R})}
	\end{equation}
	and it suffices to show that $\|\nabla s\|_{L^2(B_{2R})}\leq C\|s\|_{L^2(B_{3R})}$. We use integration by parts
	\begin{align*}
	\int_{X}\eta^2\rho_{\eps}^2|\nabla s|^2 &= \int_X \eta^2\rho_{\eps}^2 hg^{i\bj}_{\infty}\nabla_i s\nabla_{\bj}\bar{s}\omega_{\infty}^n\\
	& = -\int_X \eta^2\rho_{\eps}^2 hg^{i\bj}_{\infty}\nabla_{\bj}\nabla_i s\bar{s}\omega_{\infty}^n -2 \int_X \nabla_{\bj}(\eta^2\rho_{\eps}^2) hg^{i\bj}_{\infty}\nabla_i s\bar{s}\omega_{\infty}^n\\
	& \leq n\int_X \eta^2\rho_{\eps}^2|s|^2 + 2\int_X \eta\rho_{\eps}(\rho_{\eps}|\nabla \eta|+\eta|\nabla\rho_{\eps}|)|s||\nabla s|\\
	&\leq C\int_X (\eta^2+|\nabla \eta|^2)\rho_{\eps}^2|s|^2 + \eps\int_{X}\eta^2\rho_{\eps}^2|\nabla s|^2+ C\int_X\eta^2|\nabla \rho_{\eps}|^2|s|^2
	\end{align*}
	taking $\eps$ to $0$ gives 
	\begin{equation*}
	\int_X \eta^2|\nabla s|^2\leq C\int_X(\eta^2+|\nabla \eta|^2)|s|^2
	\end{equation*}
	by choosing $0\leq \eta\leq 1$ so that $\text{supp}(\eta)\subset B_{4R}$ and $\eta =1$ on $B_{2R}$, this gives $\|\nabla s\|_{L^2(B_{2R})}\leq \|s\|_{L^2(B_{4R})}$ Combined with estimate \eqref{gradient bound for sections}, this gives the desired estimates. 
\end{proof}

\begin{Corollary}
	For any holomorphic sections $s_0,  s_1\in H^0(L^k|_K)$ on $K$, the function $|s_i|_{h^k_{\infty}}$ extends as a lipshitz function on $K$ and this function vanishes precisely on the set $\pi_{\infty}^{-1}(\{s_i = 0\})$. Also, $\frac{s_0}{s_1}$ extends as a locally Lipshitz function defined on the set $\{|s_1|_{h^k_{\infty}}>0\}$. 
\end{Corollary}
\begin{proof}This follows immediately from Kato's inequality
	\[
		|\nabla|s|_{h^k_{\infty}}|_{g_{\infty}}\leq |\nabla s|_{h^k_{\infty}, kg_{\infty}}\leq C
	\]
	\[
		\left|\nabla\frac{s_0}{s_1}\right|_{g_{\infty}} \leq \frac{|s_1|_{h^k_{\infty}}|\nabla s_0|_{h^k_{\infty}, kg_{\infty}}+|s_0|_{h^k_{\infty}}|\nabla s_1|_{h^k_{\infty}, kg_{\infty}}}{|s_1|_{h^k_{\infty}}^2}\leq \frac{C}{|s_1|^2_{h^k_{\infty}}}
	\]
	and the fact that $K = \overline{(K\cap X_0^{reg}, g_{\infty})}$. 
\end{proof}
In this section we prove that the map $\pi_{\infty}:X_{\infty}\to X_0$ is injective, hence it is an isomorphism. 
\begin{Proposition}
	For any $p, q\in X_{\infty}$ with $p\neq q$ there exist an $k = k(p, q)>0$ and $s_p, s_q\in H^0(L^k)$ such that
	\begin{equation*}
		|s_p(p)|_{h^{k}_{\infty}}, |s_q(q)|_{h^{k}_{\infty}} \geq \frac{2}{5}
	\end{equation*}
	and 
	\begin{equation*}
		|s_p(q)|_{h^{k}_{\infty}}, |s_q(p)|_{h^{k}_{\infty}} \leq \frac{1}{3}
	\end{equation*}
\end{Proposition}
\begin{proof}
	This follows from the same argument as Proposition 3.9 in \cite{JS14}.

\end{proof}
\begin{Proposition}\label{prop: homeomorphism}
	The map $\pi_{\infty}: X_{\infty}\to X_0$ is an homeomorphism. 
\end{Proposition}
\begin{proof}
	It's clear that the map is surjective and restricts to a homeomorphism on $X_0^{reg}\subset X_{\infty}$, it suffices to show that is seperates points near $X_0^{sing}$. Given $p, q\in K$, suppose for a contradiction that $\pi_{\infty}(p)=\pi_{\infty}(q)$, then for any $k>0$, and any two sections $s_0, s_1\in H^0(K\cap X_0^{reg}, L^k)$, by the normality of $X_0$, we know that these two sections extend over the singular set to two sections of $s_0', s_1'\in H^0(\pi_{\infty}(K), L^k)$, hence we must have $\frac{s_0(p)}{s_1(p)} = \frac{s_0(q)}{s_1(q)}$. But if $d_{X_{\infty}}(p, q)>0$, then by the previous lemma, there exist $k>0$ and we can construct sections $s_p, s_q\in H^0(K\cap X_0^{reg}, L^k)$ such that $|s_p|_{h^k_{\infty}}(p), |s_q|_{h^k_{\infty}}(q) \geq\frac{2}{5}>\frac{1}{3}\geq|s_p|_{h^k_{\infty}}(q), |s_q|_{h^k_{\infty}}(p) $ which contradicts $\frac{s_p(p)}{s_q(p)} = \frac{s_p(q)}{s_q(q)}$. 
	
Observe that the singular set $\mathcal S\subset X_\infty$ is closed and of finite diameter, from which we can see that $\pi_\infty:X_\infty \to X_0$ is a proper map, hence closed, and this implies $\pi_\infty^{-1}$ is also continuous. Thus $\pi_\infty$ is a homeomorphism. 
\end{proof}

\begin{proof}[proof of Theorem \ref{Theorem2}]
	This is just a combination of Proposition \ref{Schwartz Lemma}, Proposition~\ref{prop: GH-limit-metric-completion-homeo} and Proposition~\ref{prop: homeomorphism}. 
\end{proof}

\section{Examples and Applications}\label{sec: Examples}
In this section we apply Theorems~\ref{Main Theorem} and Theorem~\ref{Theorem2} to study certain explicit examples of crepant resolutions.
\subsection{Small Resolutions of Brieskorn-Pham cones}\label{subsec: BP}

Consider the quasi-homogeneous affine varieties
\[
Y_{p,q} = \{xy +z^p-w^q=0\} \subset \mathbb{C}^4,
\]
where we assume that $p\leq q$.  These singularities, which are compound du Val of type $cA_p$ are Gorenstein and log-terminal and by the main result of \cite{CS}, $Y_{p,q}$ admits a conical Calabi-Yau metric if and only if $q<2p$.  Let $r$ denote the radial function of the Calabi-Yau cone metric. The Euler vector field $r \frac{\d}{\d r}$ associated with the cone structure is given by the real part of the holomorphic vector Reeb field $\xi$ acting on the coordinates $(x,y,z,w)$ with weights
\[
\frac{3}{2(p+q)} (pq, pq, 2q,2p);
\]
in particular, the $Y_{p,q}$ are quasi-regular Calabi-Yau cones.   A result of Katz \cite{Katz} says that the $Y_{p,q}$ admits a small (and hence crepant) resolution $\mu: Y\rightarrow Y_{p,p}$ if and only if $p=q$.  In fact, $Y_{p,p}$ admits $p$ inequivalent small resolutions
 \[
\begin{tikzcd}
Y^{1} \arrow[rrd, "\mu_{1}" ] & Y^{2} \arrow[rd, "\mu_2"]& \cdots & Y^{p-1} \arrow[ld, "\mu_{p-1}" '] & \arrow[lld, "\mu_p" '] Y^p\\
\, & \,& Y_{p,p}&\, &\,
\end{tikzcd}
\]
with each pair $Y^i, Y^j$ related by a flop; the $\frac{p(p-1)}{2}$ flops are in correspondence with the reflections in the Weyl group of the $A_{p-1}$ Dynkin diagram \cite{Matsuki2}.  When $p=2$, this recovers the Atiyah flop \cite{Atiyah}.  The exceptional locus of each contraction $\mu_{j}$ is a chain of $p-1$ rational curves with normal bundle $(-1,-1)$ intersecting transversally. Explicitly, let $\zeta = e^{\frac{2\pi\sqrt{-1}}{p}}$, and write
\[
z^p-w^p = \prod_{j=0}^{p-1}(z-\zeta^jw). 
\]
Fix $1\leq \ell \leq p-1$ and consider the rational map $\nu_{\ell}: Y_{p,p}\rightarrow \mathbb{P}^1$ defined by
\begin{equation}\label{eq: KatzRes}
\nu_{\ell}(x,y,z,w)= ([x: \prod_{j=0}^{\ell}(z-\zeta^jw)]) \in \mathbb{P}^1.
\end{equation}
Then a small resolution $ \mu: Y \rightarrow Y_{p,p}$ (say $Y^1$ for concreteness) is obtained by taking the closure of the graph of
\[
\nu_1\times \cdots\times \nu_{p-1} : Y_{p,p}\rightarrow \mathbb{P}_{(1)}^1 \times\cdots \times  \mathbb{P}_{(p-1)}^1.
\]
There are also corresponding partial resolutions $\overline{Y}$ by projecting out some collection of the $\nu_{j}$. Fix $1\leq i \leq p$, and let $\overline{Y}$ be any partial resolution whose contraction $\bar{\pi}: \overline{Y}\rightarrow Y_{p,p}$ factors through $\nu_i$.   Clearly these resolutions are obtained by repeatedly blowing-up along the lines $x= z-\zeta^jw=0$.   There is a divisor $E_i$ defined by $-E_i = \nu_i^{-1}(p)$ for a generic point $p\in \mathbb{P}^1$, and these divisors satisfy $\mathcal{O}_{\overline{Y}}(-E_i)\big|_{{\rm Exc}(\nu_i)} = \mathcal{O}_{\mathbb{P}^1}(1)$, and $\mathcal{O}_{\overline{Y}}(E_i)$ is trivial on any other component of ${\rm Exc}(\bar{\pi})$.   Furthermore, if $\overline{Y}$ is obtained from $\nu_{i_1}\times\cdots \times \nu_{i_k}$ then $\bigotimes_{j=1}^k \mathcal{O}_{\overline{Y}}(-E_{i_j})^{\otimes \ell_j}$ is ample for any $\ell_{j} \in \mathbb{Z}_{>0}$.  These statements follows straightforwardly from the corresponding statements for the blow-ups of the ambient $\mathbb{C}^4$.

Let us fix a small resolution $\mu:Y \rightarrow Y_{p,p}$.  By Hartog's theorem the holomorphic Reeb vector field extends over ${\rm Exc}(\mu)$ and generates a holomorphic retraction onto ${\rm Exc}(\mu)$.  Thus we have
\[
H^{1,1}(Y,\mathbb{R}) = \bigoplus_{i=1}^{p-1}H^{1,1}(\mathbb{P}_{(i)}^1, \mathbb{R}) = \bigoplus_{i=1}^{p-1}\mathbb{R}\cdot [E_i]
\]
By the above discussion, the classes $\sum_{i=1}^{p-1}(-t_i)[E_i]$ are K\"ahler on $Y$, provided $t_i>0$ for all $i$, and semi-positive for $t_i \geq 0$.  Each of these cohomology classes is $2$-almost compactly supported.  Fix a class $[\alpha_0] =\sum_{i=1}^{p-1}(-t_i)[E_i]$ where $t_i\geq 0$, and at least one $t_j=0$ and let $[\omega] \in H^{1,1}(Y,\mathbb{R})$ be any K\"ahler class.  Let $[\omega_t]= (1-t)[\alpha_0] + t[\omega]$ be a linear family of K\"ahler classes. Then by \cite{Go} (see also \cite{CH13}) there is an asymptotically conical Calabi-Yau metric $\omega_{t,CY}$ in $[\omega_t]$ for all $t>0$.

Since the cone at infinity is quasi-regular we can apply Lemma~\ref{lem: DP} to conclude that there is a K\"ahler current in $[\alpha_0]$ which is smooth on the complement of
\[
V:= \left\{ \mathbb{P}^1_{j} \subset Y: \int_{\mathbb{P}^1_{(j)}} \alpha_0 =0\right\}
\]
Let $\overline{Y}$ be the partial resolution obtained by contracting $V$, and let $\hat{\pi}: Y \rightarrow \overline{Y}$ be the contraction map.  If $[\alpha_0] \in H^{1,1}(Y,\mathbb{Q})$ then, by the preceding discussion,  after rescaling we can assume that $[\alpha_0] = \pi^*c_1(L)$ for some ample line bundle $L\rightarrow \overline{Y}$.  Applying Theorem~\ref{Main Theorem} and Theorem~\ref{Theorem2} we obtain

\begin{Proposition}
In the above situation we have
\begin{enumerate}
\item $\overline{Y}_{reg}$ admits a smooth Ricci-flat metric $\bar{\omega}$, asymptotic to the Calabi-Yau metric on $Y_{p,p}$ at infinity, and with $\overline{(\overline{Y}_{reg}, \bar{\omega})}$ homeomorphic to $\overline{Y}$.
\item As $t\rightarrow 0$ $(Y,\omega_{t, CY})$ converges in the Gromov-Hausdorff sense to $\overline{(\overline{Y}_{reg}, \bar{\omega})}$.
\item In particular, if we take $[\alpha_0]=0$, the flops of the $Y_{p,p}$ are continuous in the Gromov-Hausdorff sense.
\end{enumerate}
\end{Proposition}

\begin{proof}
The only point which is not an immediate consequence of Theorems~\ref{Main Theorem} and~\ref{Theorem2} is the third point.  However, by the uniqueness part of Theorem~\ref{Main Theorem}, the limiting limiting Calabi-Yau metric $\bar{\omega}$ on $Y_{p,p}$ is isometric to the conical Calabi-Yau metric from \cite{CS}.  Alternatively, this can be seen as follows.  Let $\omega_{c}$ denote the Calabi-Yau metric on $Y_{p,p}$.  Clearly $t\omega_{1,CY}$ is a Calabi-Yau metric in $t[\omega]$ asymptotic to $t\omega_{c}$.  Let $\hat{\xi}$ denote the extension of the holomorphic Reeb vector field on $Y$, and, for $\lambda \in \mathbb{C}$ let $\phi_{\lambda}:Y\rightarrow Y$ denote the $\lambda$-flow of $\hat{\xi}$.  Then
\[
\left(\phi_{\frac{1}{\sqrt{t}}}\right)^*t\omega_{1,CY}
\]
is Calabi-Yau, asymptotic to $\omega_c$, and lies in the cohomology class $t[\omega]$ and hence is equal to $\omega_{t,CY}$ by the uniqueness results of \cite{CH13}.  From this description, and the convergence result of Theorem~\ref{Main Theorem} it follows that $\omega_{t,CY}$ converges to $\mu_{i}^*\omega_{c}$ on compact sets of $Y\backslash {\rm Exc}(\mu_i)$.  
\end{proof}

It's not hard to check that if a partial resolution $\overline{Y}$ is obtained by blowing up $0<k<p-1$ lines $x= z-\zeta^jw=0$, then $\overline{Y}$ has an isolated singularity biholomorphic to a neighborhood of the singular point in $Y_{p-k, p-k}$.  More precisely, suppose for simplicity that $\overline{Y}$ is obtained by blowing-up the lines $x= z-\zeta^jw=0$ for $0\leq j\leq k <p-1$.  Then $\overline{Y}$ has an isolated singularity biholomorphic to
\[
\widetilde{Y}_{p-k,p-k}:=\{ xy= \prod_{j=k}^{p-1}(z-\zeta^jw)\} \subset \mathbb{C}^4
\]
which is deformation equivalent to $Y_{p-k,p-k}$ and admits a conical Calabi-Yau metric by argument of \cite{CS}.  The link of this singularity is topologically $(p-k-1)\#(S^2\times S^3)$ and it comes equipped with a Sasaki-Einstein metric.  It was shown in \cite{CS} that the volume of these Sasaki-Einstein metrics is given by
\[
\frac{2 (2(p-k))^3}{27 (p-k)^4} = \frac{16}{27(p-k)}
\]

   Thus $\overline{Y}$ yields a cobordism between $(p-k-1)\#(S^2\times S^3)$ and $\#(p-1)(S^{2}\times S^3)$.  It is natural to expect that that the metric $\bar{\omega}$ on $\overline{Y}$, close to the singular point, is close to the conical Calabi-Yau metric on $\tilde{Y}_{p-k,p-k}$.  At the very least, we expect

\begin{Conjecture}\label{conj: Ypk}
Let $(\overline{Y},d)$ denote the metric space obtained as the completion of $(\overline{Y}_{reg}, \bar{\omega})$.  Then the tangent cone to $(\overline{Y},d)$ at the singular point is isometric to $\widetilde{Y}_{p-k,p-k}$ equipped with its conical Calabi-Yau metric.
\end{Conjecture}

Let $y \in \overline{Y}$ denote the singular point, and consider the function
\[
\mathbb{R}_{>0} \ni r \mapsto v(r) := \frac{{\rm Vol}_{\bar{\omega}}(B_{\bar{\omega}}(y, r))}{r^6}
\]
Since $(\overline{Y}, \overline{\omega})$ is Calabi-Yau, $v(r)$ is monotone decreasing by the Bishop-Gromov comparison theorem.  Furthermore, assuming Conjecture~\ref{conj: Ypk}, since $\bar{\omega}$ is asymptotic to the conical Calabi-Yau metric on $Y_{p,p}$ we have
\[
 \frac{16}{27(p-k)}=\lim_{r\rightarrow 0} v(r) \geq \lim_{r\rightarrow +\infty} v(r) =  \frac{16}{27p}.
\]
Note that the equality case of Bishop-Gromov already shows that if $k=0$, then the metric is conical.  

While deducing $k\geq 0$ in this way is not particularly interesting, this discussion holds for any asymptotically conical Calabi-Yau variety with or without singularities (indeed, a smooth, asymptotically conical Calabi-Yau variety is naturally a cobordism between the standard Sasaki-Einstein structure on the sphere and the link of the cone at infinity).  Suppose $(\overline{Y}, \bar{\omega})$ is a asymptotically conical Calabi-Yau variety with asymptotic cone $C_{\infty}$, and with a singular point $y$.  Assume that a neighborhood of $y$ is biholomorphic to a neighborhood of an isolated singular point in some quasi-homogeneous affine variety $C_0$ admitting a conical Calabi-Yau metric.  Assuming that $\bar{\omega}$ is close to the Calabi-Yau metric on $C_0$ near the singularity at $y$,  the volume ratio of geodesic balls centered at $y$ will decrease (by Bishop-Gromov) from the volume ratio of the cone $v(C_0$) to the volume of ratio of the cone at infinity, $v(C_{\infty})$.  Since these volume ratios are algebraic invariants of the singularities $C_0, C_{\infty}$, this situation is obstructed in general; for example one cannot take $C_0 = Y_{p,p}$ and $C_{\infty} = Y_{p-k, p-k}$.

It is tempting to speculate that the volume function on Sasaki-Einstein structures could give rise to a sort of Morse function on the space of Sasaki-Einstein manifolds.  For two Sasaki-Einstein manifolds $S_0, S_\infty$ with corresponding cones $C_0, C_\infty$ a Calabi-Yau space $(\overline{Y}, \bar{\omega})$ with an isolated singularity $C_0$ and cone $C_\infty$ at infinity could be regarded as a kind of flow line of the Morse function between $S_0$ and $S_\infty$.  We will give further examples of this discussion below.

\subsection{Examples from Fano manifolds}

Let us next indicate how to construct examples starting from Fano manifolds with a different singular structure than the previous examples.  Suppose $X$ is a Fano manifold of dimension $n$.  Let $\tilde{X}= \Bl_{p}X$ be the blow up of $X$ at a point and let $\tilde{E}\subset \tilde{X}$ be the exceptional divisor.  Assume in addition that that $\tilde{X}$ is Fano and $-K_{\tilde{X}}$ is base-point free.  Assume that $\tilde{X}$ has a K\"ahler-Einstein metric, or more generally that the affine cone over  $\tilde{X}$, $\Spec \bigoplus_{m\geq 0} H^{0}(\tilde{X}, -K_{\tilde{X}}^{\otimes m})$, admits a conical Calabi-Yau metric.  This holds, for example, whenever $\tilde{X}$ is toric, by \cite{FOW}.  It is not difficult to generate examples satisfying these assumptions. For example
\begin{itemize}
\item Let $X = \mathbb{P}^n$, with $p$ a torus invariant point.  Then $\tilde{X} = \rm Bl_p\mathbb{P}^n$ is Fano and $-K_{\tilde{X}}$ is base point free.  These manifolds do not admit K\"ahler-Einstein metrics, as can be seen from Matsushima's obstruction.  However, they are toric, and so the theorem of Futaki-Ono-Wang implies the existence of a Calabi-Yau cone metric on the affine cone $C:= \Spec \bigoplus_{m\geq 0} H^{0}(\tilde{X}, -K_{\tilde{X}}^{\otimes m})$.  Note that the conical Calabi-Yau structure on $C$ need not be quasi-regular, as happens for example when $n=2$ \cite{GMSW, FOW, MSYau}.  

\item Let $X$ be a del Pezzo surface with $K_{X}^2 \geq 3$, and $p$ chosen sufficiently generic so that $\tilde{X}= \Bl_p X$ is Fano.  The global generation of $-K_{\tilde{X}}$ follows from Reider's Theorem \cite{Reider}.  Furthermore, a theorems of Tian-Yau \cite{TY87} and Tian \cite{Tian90} say that $X$ admits a K\"ahler-Einstein metric if $K_{X}^2 <8$.  If, however, $K_{X}^2=8,9$ then $X$ does not admit a K\"ahler-Einstein metric by Matsushima's obstruction \cite{Mats}.  On the other hand, in these latter examples, if $p$ is chosen so that $\tilde{X}$ is toric, then the affine cone $\Spec \bigoplus_{m\geq 0} H^{0}(\tilde{X}, -K_{\tilde{X}}^{\otimes m})$ admits a conical Calabi-Yau metric thanks to results of Futaki-Ono-Wang \cite{FOW} (See also \cite{CS}).  In these examples the Calabi-Yau cone structure is not quasi-regular \cite{MSYau, FOW}.
\end{itemize}

Let $Y= K_{\tilde{X}}$ be the total space of the canonical bundle, and let $p: Y\rightarrow \tilde{X}$ be the projection.  The pull-back $p^*$ identifies $H^{1,1}(Y,\mathbb{R})= H^{1,1}(\tilde{X},\mathbb{R})$, and $Y$ admits an asymptotically conical Calabi-Yau metric in any K\"ahler class in $H^{1,1}(Y,\mathbb{R})$ \cite{CH13}.  Suppose $[\alpha]\in H^{1,1}(X,\mathbb{R})$ is a K\"ahler class, so that $p^*[\pi^*\alpha] \in H^{1,1}(Y,\mathbb{R })$ is a nef class on $Y$ admitting a semi-positive representative.  By regarding the exceptional divisor of the blow-up $\pi: \tilde{X}\rightarrow X$ as a subvariety of the zero section in $Y$, we get a natural codimension $2$ subvariety, $E \subset Y$ (explicitly $E = p^{-1}(\tilde{E})\cap \{\text{ zero section }\}$).  Our goal is to show that if $[\omega_t] =(1-t)[p^*\pi^*\alpha]+t[\omega]  \in H^{1,1}(Y,\mathbb{R})$ and $\omega_{t,CY}$ are conical Calabi-Yau metrics in $[\omega_t]$ then, as $t\rightarrow 0$, $(Y,\omega_{t,CY})$ Gromov-Hausdorff converges to a variety $Z$ with an isolated, Gorenstein, log-terminal singularity which is obtained from $Y$ by contracting $E$ to a point.  As a first step, we need to verify that Assumption~\ref{ass: mainAss} holds, since the failure of the cone at infinity to be quasi-regular means that Lemma~\ref{lem: DP} does not apply in general.

\begin{Lemma}\label{lem: exampleAss1}
The cohomology class $p^*[\pi^*\alpha]$ contains a K\"ahler current which is smooth outside of $E$. 
\end{Lemma}
\begin{proof}
It is a standard fact that we can choose a hermitian metric $h_{\tilde{E}}$ on $\mathcal{O}_{\tilde{X}}(\tilde{E})$ such that
\begin{equation}\label{eq: KahlerCurrentX}
\pi^*\alpha + \epsilon \ddb \log h_{\tilde{E}} >  \omega_{\tilde{X}}
\end{equation}
for some $\epsilon>0$ and $\omega_{\tilde{X}}$ a K\"ahler form on $\tilde{X}$.  Let $s_{\tilde{E}}$ denote the defining section of $\tilde{E}\subset \tilde{X}$.  After scaling we may assume that $|s_{\tilde{E}}|^2_{h_{\tilde{E}}} <1$.  The current $\tilde{T} := \pi^*\alpha + \epsilon \ddb \log|s_{\tilde{E}}|^2_{h_{\tilde{E}}}$ is a K\"ahler current on $\tilde{X}$ which is singular along $\tilde{E}\subset \tilde{X}$.  Let $h_{\tilde{X}}$ be a negatively curved metric on $K_{\tilde{X}}$, and let $s$ denote a coordinate on the fibers of $K_{\tilde{X}}$.  We claim that
\begin{equation}\label{eq: KahlerCurrentY}
T= p^*\pi^*\alpha + \ddb (|s|^2_{h} +\epsilon \log(p^*|s_{\tilde{E}}|_{h_{\tilde{E}}}^2 + |s|^2_{h_{\tilde{X}}}))
\end{equation}
 is a K\"ahler current.  This can be verified by a straightforward calculation, which we leave to the reader.

\end{proof}

The next step is to show that there is a space $Z$, and a map $\Phi:Y\rightarrow Z$ which is an isomorphism outside $E$ and contracts $E$ to a point, which is an isolated, Gorenstein log-terminal singularity in $Z$.  Let us begin with a local description of this map and the resulting singularity.  Note that the normal bundle of $E\subset Y$ is given by
\[
N_{E/Y} = \mathcal{O}_{\mathbb{P}^{n-1}}(-1) \oplus \mathcal{O}_{\mathbb{P}^{n-1}}(-(n-1))
\]
which follows from $K_{\tilde{X}} = \pi^*K_{X} + (n-1)\tilde{E}$.  There is a contraction map
\[
\nu:\mathcal{O}_{\mathbb{P}^{n-1}}(-1) \oplus \mathcal{O}_{\mathbb{P}^{n-1}}(-(n-1))\rightarrow C
\]
 contracting the zero section of $N_{E/Y}$ to a point.  Explicitly, this map is given by \cite[Page 314]{Matsuki}
 \[
 \begin{aligned}
 N_{E/Y} &= {\rm Spec} \bigoplus_{m\geq 0} {\rm Sym}^{m} \left(\mathcal{O}_{\mathbb{P}^{n-1}}(1) \oplus \mathcal{O}_{\mathbb{P}^{n-1}}((n-1))\right)\\
 & \rightarrow {\rm Spec} \bigoplus_{m\geq 0} H^{0}\left(\mathbb{P}^{n-1}, {\rm Sym}^{m}\left( \mathcal{O}_{\mathbb{P}^{n-1}}(1) \oplus \mathcal{O}_{\mathbb{P}^{n-1}}((n-1))\right)\right) = C_0
 \end{aligned}
 \]
Since 
\[
H^{0}\left(\mathbb{P}^{n-1}, {\rm Sym}^{m}\left( \mathcal{O}_{\mathbb{P}^{n-1}}(1) \oplus \mathcal{O}_{\mathbb{P}^{n-1}}((n-1))\right)\right) = H^{0}(\mathbb{P}(N_{E/Y}), \mathcal{O}_{\mathbb{P}(N_{E/Y})}(m))
\]
we see that $C_0$ is the affine cone over $\mathbb{P}(N_{E/Y})$ obtained by blowing down the zero section of $\mathcal{O}_{\mathbb{P}(N_{E/Y})}(-1)$.  We claim that $\mathbb{P}(N_{E/Y})$ is Fano.  In general, the canonical bundle of a projective bundle $\pi: \mathbb{P}(V) \rightarrow X$, where $V$ has rank $r$ is given by
\[
K_{\mathbb{P}(V)} = \mathcal{O}_{\mathbb{P}(V)}(-r-1) \otimes \pi^*(\det V^*) \otimes \pi^*K_{X}.
\]
Applying this formula in the current scenario yields
\[
K_{\mathbb{P}(N_{E/Y})} = \mathcal{O}_{\mathbb{P}(N_{E/Y})}(-3).
\]
Since $N_{E/Y}$ is a direct sum of negative line bundles,  $\mathcal{O}_{\mathbb{P}(N_{E/Y})}(3)$ is ample.  It follows from this that $C_0$ has an isolated Gorenstein, log-terminal singularity and $K_{C_0} \sim \mathcal{O}_{C_0}$ is trivial. Finally, since $N_{E/Y}\rightarrow \mathbb{P}^{n-1}$ is a direct sum of line bundles, $\mathbb{P}(N_{E/Y})$ is toric.  Therefore the result of Futaki-Ono-Wang \cite{FOW} (see also \cite{CS}) says that $C_0$ admits a conical Calabi-Yau metric for some choice of Reeb vector field.  

Next we will globalize this construction using the input of an ample line bundle $L$ on $X$. First note that a section  $f \in H^{0}(\tilde{X}, -K_{\tilde{X}}^{\otimes m})$ naturally induces a holomorphic function $f \in H^{0}(Y, \mathcal{O}_{Y})$ vanishing to order $m$ on $\tilde{X} = \{\text{ zero section }\} \subset Y$.  Let $f_1\ldots, f_{M}$ be generators of the coordinate ring $\bigoplus_{m>0}H^{0}(\tilde{X}, -K_{\tilde{X}}^{\otimes m})$. Since $-K_{\tilde{X}}$ is ample and globally generated, the holomorphic functions $f_1\ldots, f_{M}$ separate points and tangent vectors on $Y\backslash \tilde{X}$, and generate the normal bundle to $\tilde{X}$ in $Y$.  Let $L$ be a very ample line bundle on $X$, and let $\{s_0,\ldots, s_{N}\}$ be a basis of $H^{0}(X, L)$.  Fix coordinates $(z_1,\ldots,z_n)$ on $X$ centered at $p$.  Up to making a linear change of coordinates we can assume that $s_0(p)\ne0$, and near $p$ we have
\[
\frac{s_i(z)}{s_0(z)}  = z_i +O(z^2)\quad  1\leq i \leq n, \qquad \frac{s_j(z)}{s_0(p)} = O(z^2) \quad n\leq j\leq N
\]
By inspection the sections $\{p^*\pi^*s_i\}_{0\leq i \leq N}$ separate points and tangents in $\tilde{X}\backslash \tilde{E}$ and generate the normal bundle to $\tilde{E}$ in $\tilde{X}$.  Now consider the map $\Phi: Y \rightarrow \mathbb{P}^N\times \mathbb{P}^{M}$ defined by
\begin{equation}\label{eq: mapPhi}
\Phi(z):= ([p^*\pi^*s_0(z): \cdots :p^*\pi^*s_N(z)],  [1: f_1(z):\cdots:f_M(z)]) \in \mathbb{P}^N\times \mathbb{P}^{M}
\end{equation}
By the preceeding discussion this map is an isomorphism on $Y\backslash E$, and $\Phi(E) = [1:0:\cdots:0] \times [1:0:\cdots:0]$. Since the differential $d\Phi$ is an isomorphism on $N_{E/Y}$, the germ of $\Phi$ agrees with the contraction $\nu$ on $N_{E/Y}$.  Note also that $\Phi\big|_{\tilde{X}} = \pi$ (composed with the imbedding $X$ into projective space by sections of $L$). Let $Z= \Phi(Y)$.  From the local description above $Z$ has an isolated Gorenstein, log-terminal singularity, and $K_{Z} = \mathcal{O}_{Z}$.  The map $\Phi:Y\rightarrow Z$ is therefore a small, and hence crepant, resolution of $Z$.  It follows from the construction that we can describe $Z$ has the relative spectrum
\[
Z= \underline{\Spec}(K_{X}\otimes \mathfrak{m}_{p}) \rightarrow X
\]
where $\mathfrak{m}_p$ is the ideal sheaf of $p\in X$. In order to apply Theorems~\ref{Main Theorem} and~\ref{Theorem2} it suffices to show

\begin{Lemma}\label{lem: pullBackZ}
In the above setting, there is an ample line bundle $L'$ on $Z$ such that $p^*c_1(\pi^*L) = \Phi^*c_1(L')$.
\end{Lemma}

\begin{proof}
Since $Z$ is normal and $\Phi$ is projective with connected fibers we have $\Phi_{*}\mathcal{O}_{Y} = \mathcal{O}_{Z}$, and $f_1,\ldots, f_{M}$ extend over the singular point to global sections of $\mathcal{O}_{Z}$.  Furthermore, there is a natural projection
\[
\hat{p}:Z\rightarrow X
\]
obtained by projecting from $Z$ onto the $\mathbb{P}^N$ factor in~\eqref{eq: mapPhi} and we have $\pi\circ p =\Phi \circ p = \hat{p}\circ \Phi$.  Thus
\[
[p^*\pi^*s_0: \cdots :p^*\pi^*s_N] = [\hat{p}^*s_0: \cdots :\hat{p}^*s_N].
\]
Combining this observation with the Segre embedding $\mathbb{P}^N\times \mathbb{P}^{M}\hookrightarrow \mathbb{P}^{(N+1)(M+1)-1}$ it follows that $L':=\hat{p}^*L$ is ample on $Z$.  Since
\[
p^*c_1(\pi^*L) = \Phi^*c_1(\hat{p}^*L)
\]
the lemma follows.
\end{proof}

We can now conclude

\begin{Corollary}
With notation as above, consider the family of K\"ahler classes $[\omega_t] = (1-t)p^*c_1(\pi^*L) + t[\omega] \in H^{1,1}(Y,\mathbb{R})$ for $t>0$.  Let $\omega_{t,CY}$ be the asymptotically conical K\"ahler metrics in $[\omega_t]$.  Then there is a incomplete, asymptotically conical Calabi-Yau metric $\overline{\omega}$ on $Z_{reg}$ such that $\overline{(Z_{reg} ,\bar{\omega})} = (Z, d)$ and 
\[
(Y, \omega_{t, CY}) \rightarrow_{GH} (Z,d).
\]
\end{Corollary}

\begin{proof}
Combine Lemmas~\ref{lem: exampleAss1}~\ref{lem: pullBackZ} with Theorems~\ref{Main Theorem} and~\ref{Theorem2}.
\end{proof}

It is again natural to conjecture

\begin{Conjecture}
Let $(Z,d)$ be the metric space structure on $Z$ induced from $Y$ by Theorem~\ref{Theorem2}.  Then the tangent cone to $(Z,d)$ at the singular point $z \in Z$ is isometric to the blow down of the zero section in $\mathcal{O}_{\mathbb{P}(V)}(-1)$ where
\[
 V := \mathcal{O}_{\mathbb{P}^{n-1}}(-1) \oplus \mathcal{O}_{\mathbb{P}^{n-1}}(-(n-1))
 \]
 equipped with its conical Calabi-Yau metric. 
 \end{Conjecture}

Assuming this conjecture, the space $Z$ can be viewed as a kind of cobordism between Sasaki-Einstein manifolds, and the speculative discussion from Section~\ref{subsec: BP} can be applied in the same way.

\end{document}